\documentclass{article}
\usepackage{amsfonts,amsmath,amsthm,amssymb,authblk,latexsym,bm,mathrsfs}
\usepackage[normalem]{ulem}

\usepackage{setspace}
\usepackage{citeref}
\usepackage{hyperref} \hypersetup{colorlinks=true}

\theoremstyle{plain}
\newtheorem{theorem}{Theorem}[section]
\newtheorem{corollary}[theorem]{Corollary}

\newtheorem{lemma}[theorem]{Lemma}
\newtheorem{definition}[theorem]{Definition}
\newtheorem{example}[theorem]{Example}

\newtheorem{remark}[theorem]{Remark}

\numberwithin{equation}{section}

\def\wh{\widehat}
\def\F{{\Bbb F}} \def\E{{\Bbb E}}  \def\C{{\Bbb C}}
\def\al{\alpha} \def\be{\beta} \def\ga{\gamma} \def\lam{\lambda}
\def\scr{\mathscr}

\begin{document}

\title{\bf Sharp Uncertainty Principle for Transitive
$G$-Sets over Arbitrary Fields and Finite Groups}

\author{
Bocong Chen$^1$, Yun Fan$^2$  and Gaojun Luo$^3$\footnote{Email addresses:
mabcchen@scut.edu.cn (B. Chen),
yfan@ccnu.edu.cn (Y. Fan)
 and
gaojun\_luo@nuaa.edu.cn (G. Luo).
}}

\date{\small
$1.$ School of Mathematics, South China University of Technology, Guangzhou  510641, China\\
$2.$  School of Mathematics and Statistics, Central China
Normal University, Wuhan 430079, China\\
$3.$ School of Mathematics, Nanjing University of Aeronautics and Astronautics, Nanjing 210016, China\\
}

\maketitle

\begin{abstract}
For any finite group $G$, any transitive $G$-set $X$ and any field $\F$,
we consider the vector space $\F^X$ of all functions from $X$ to $\F$,
which is a $G$-space isomorphic to the permutation $\F G$-module $\F X$.
When the group algebra $\F G$ is semisimple and split over $\F$,
we find a specific basis $\wh X$ of $\F^X$ and construct the
Fourier transform   $\F^X\to \F^{\wh X}$, $f\mapsto\wh f$.
We define the rank support $\mbox{rk-supp}(\wh f\,)$ and
prove that $\mbox{rk-supp}(\wh f)=\dim \F G f$,
where $\F G f$ is the submodule of the permutation module $\F X$
generated by the element $f=\sum_{x\in X}f(x)x$.
Next, we extend and strengthen  the sharpened uncertainty principle
for finite abelian groups,
originally established by Feng, Hollmann, and Xiang in 2019,
to a broader framework and a sharp version.
For any field $\F$, any transitive $G$-set $X$ and $0\ne f\in\F^X$,
we construct a block $X_{{\rm supp}(f)}$ of $X$ and a subset ${\scr S}'^{-\!1}$ of $G$
determined by the support ${\rm supp}(f)$ of $f$, and show that
$\dim\F Gf-\dim\F{\scr S}'^{-\!1}\!f\ge 1$ and
$$
  |{\rm supp}(f)|\cdot \dim\F Gf
  \ge |X|+ (\!\dim\F Gf-\dim\F{\scr S}'^{-\!1}\!f)
  \cdot|{\rm supp}(f)|  -|X_{{\rm supp}(f)}|,
$$
where ${\scr S}'^{-\!1}\!f=\{\al f\,|\,\al\in{\scr S}'^{-\!1}\}\subseteq\F X$,
and $\F{\scr S}'^{-1}\!f$ denotes the subspace of $\F X$ spanned by the subset
${\scr S}'^{-\!1}\!f$.
Furthermore, we provide necessary and sufficient
conditions for the above inequality to achieve equality.
As corollaries, we derive many sharpened or
classical versions of the finite-dimensional uncertainty principle,
in particular  addressing an open question
posed by Feng, Hollmann, and Xiang.
When $G$ is  of prime order and $X=G$, we give a lower bound on
$\dim\C Gf$ that recovers Tao's 2005 strong uncertainty principle, along with a precise characterization of the equality case in this scenario.

\medskip
{\bf Key words}: Finite group; group action;
support of function; Fourier transform; uncertainty principle.


\end{abstract}

\section{Introduction} \label{introduction}
Let $\F$ be a field and let $\C$ be the field of complex numbers.
Given a set $X$, denote its cardinality by $|X|$.
Let $\F^X$ denote the $\F$-vector space of all functions from $X$ to~$\F$.
For $f \in \F^X$, the support of $f$ is defined as
$${\rm supp}(f)=\{x\,|\,x\in X,\,f(x)\ne 0\}.$$
Throughout this paper, $G$ denotes a finite group.
We write $H \le G$ to indicate that $H$ is a subgroup of $G$.
The group algebra $\F G$ is the $\F$-vector space with basis $G$, equipped with multiplication defined by the group operation on $G$.
Let ${\rm Irr}(G)$ denote the set of all absolutely irreducible characters of $G$ (possibly taking values in an extension of $\F$).
For $\psi \in {\rm Irr}(G)$, let $n_\psi$ denote the degree of~$\psi$.

The term
 ``uncertainty principle" encompasses
several important theorems across various areas of
mathematics and physics, all asserting that there is
a trade-off between the localization of a function
and that of its Fourier transform.
The uncertainty principle in mathematics,
particularly in the context of finite groups and the Fourier transforms,
plays a significant role in understanding the
relationship between a function and its Fourier transform.
As noted in \cite{WW21}, the ``uncertainty principle''
refers to a class of theorems stating that a non-zero function and its Fourier transform
cannot both have small supports simultaneously.

The celebrated Donoho-Stark uncertainty principle in \cite{DS89}, \cite{Te99}
applies to a finite abelian group $G$
%{\red and its dual group $\wh G={\rm Irr}(G)$}
with $\F = \C$.
For each nonzero function $f \in \C^G$, its Fourier transform $\widehat{f} \in \C^{\widehat{G}}$ is defined by
\begin{align} \label{eq wh f=}
 \widehat{f}(\psi) = \sum_{x \in G} f(x) \psi(x),
\quad \text{for all } \psi \in \widehat{G},
\end{align}
where $\widehat{G} = {\rm Irr}(G)$ is the dual group of $G$.
This principle states that
\begin{align} \label{u p DS}
 \big|{\rm supp}(f)\big|\!\cdot\!\big|{\rm supp}(\wh f)\big| \ge |G|,
\end{align}
and the equality occurs if and only if $f=c\chi I_{\ga H}$, where $c\in\C$,
 $\chi\in{\rm Irr}(G)$, $H\le G$, $\ga\in G$ and $I_{\ga H}$ is the
indicator function of the coset $\ga H$ defined by
$I_{\ga H}(\al) =\left\{\!\!
\begin{array}{ll}1, &\al\in\ga H, \\ 0, & \al\notin\ga H.\end{array}\right.$

Tao \cite{Tao05}, along with Goldstein, Guralnick, and Isaacs \cite{GGI05}, independently established a stronger form of the
uncertainty principle, which we refer to in this paper as the {\em strong uncertainty principle}.
This principle was also independently explored by Bir\'{o} \cite{Bir} and Meshulam \cite{Me06}.
See \cite{Tao05} for details on the origin of this result. Specifically, if $G$ is a cyclic group of prime order $p$,
and $0 \neq f \in \C^G$ with $\widehat{f} \in \C^{\widehat{G}}$ being the Fourier transform of $f$, then
\begin{align} \label{s u p Tao}
\big|{\rm supp}(f)\big| + \big|{\rm supp}(\widehat{f})\big| \ge p+1.
\end{align}
A key aspect of the proof in \cite{Tao05} relies on a result by Chebotar\"ev concerning the minors
of the Fourier matrix $\big(\omega^{ij}\big)_{0 \leq i,j \leq p-1}$,
where $\omega$ is a primitive $p$-th root of unity (cf. \cite{Chebotarev,SM96}).
Notably, the conditions for equality in Equation \eqref{s u p Tao} were
not explicitly addressed in \cite{Bir,GGI05,Me06,Tao05}.
Garcia, Karaali, and Katz \cite{Ga} generalized the strong uncertainty principle, establishing a more robust version for $\chi$-symmetric functions.

The aforementioned studies focus on complex-valued functions.
Wigderson and Wigderson \cite{WW21} presented an elegant and unified approach to these results by utilizing various norms of linear operators.

The uncertainty principle and algebraic coding theory share a deep yet subtle historical connection. Donoho and Stark's seminal work explicitly employed the BCH bound, which is a fundamental result in algebraic coding theory for cyclic codes,
to establish a discrete uncertainty
principle for finite abelian (specifically cyclic) groups.
In coding theory, the finite field $\F$ serves as the alphabet,
while a word indexed by a finite group $G$ corresponds to a function $f \in \F^G$, which is in turn identified with the element $\sum_{x\in G}f(x)x\in\F G$.
A desirable codeword $f$ typically exhibits both a large Hamming weight
 and a large dimension $\dim\F Gf$ (cf. \cite{HP}),
where $\F Gf$ denotes the left ideal of $\F G$ generated by $f$.
The Hamming weight of the word $f$ is precisely
the cardinality $\big|{\rm supp}(f)\big|$ of the support of $f$.
On the other hand,
for a finite abelian group $G$ and a field $\F$ such that $\F G$ is semisimple
(equivalently, the conditions in Equation \eqref{condition semisimplicity} hold),
the Fourier transform  in Equation \eqref{eq wh f=} and
the uncertainty principle in Equation \eqref{u p DS}
still work well, and it is well-known that
\begin{align} \label{eq dim FGf}
\dim \F Gf=
\big|{\rm supp}(\wh f)\big|,\quad \mbox{for any}~ f\in \F^G.
\end{align}
Building on this perspective, recent works in \cite{BS22, EKL17} have leveraged the uncertainty principle to characterize and construct good codes.

Feng, Hollmann, and Xiang \cite{FHX19} provided a refined proof of the shift bound for abelian codes and subsequently generalized the Donoho-Stark uncertainty principle.
Let $G$ be an abelian group and let $\F$ be a field such that $\F G$ is semisimple.
For any nonzero $f \in \F^G$, they defined the stabilizer of~${\rm supp}(f)$ in~$G$ as:
\[
 G_{{\rm supp}(f)}=
  \big\{\,\al\;\big|\;\al\in G,\,\al\,{\cdot}\,{\rm supp}(f)={\rm supp}(f)\big\}.
\]
Notably, $G_{{\rm supp}(f)} \leq G$, and ${\rm supp}(f)$ decomposes
into a disjoint union of some cosets of $G_{{\rm supp}(f)}$
(our notation differs slightly from \cite{FHX19}).
They established the following sharpened uncertainty principle:
\begin{align}
\label{u p FHX} \big|{\rm supp}(f)\big| \cdot \big|{\rm supp}(\widehat{f})\big| \geq |G| + \big|{\rm supp}(f)\big| - |G_{{\rm supp}(f)}|.
\end{align}
From this inequality, both the Donoho-Stark uncertainty principle in Equation~\eqref{u p DS}
and its equality conditions readily follow.
As noted in \cite{FHX19}, while the equality case for the Donoho-Stark principle was fully characterized, characterizing the equality conditions for Equation \eqref{u p FHX} remains challenging. This question was explicitly left as an open problem for future research.

The initial motivation of this paper is to extend the sharpened uncertainty principle in Equation \eqref{u p FHX} to a more refined form and to study the conditions under which equality holds. Our discussion will be carried out in a broader framework.
%Let us introduce the extensive framework.

%In this paper, we extend and strengthen the sharpened uncertainty principle
%to arbitrary transitive $G$-sets $X$ associated with any finite group $G$ and any field $\F$,
%removing the restrictions imposed by Equation \eqref{condition semisimplicity}.
%We establish a sharper uncertainty principle that applies to all
%transitive $G$-sets and provide complete characterization of the equality conditions.
% This development directly resolves the open question concerning equality cases
%that was left unresolved in \cite{FHX19}.

Turning to the general case where $G$ is not necessarily abelian,
the Fourier transform  and its
inversion on the group algebra $\C G$ are described in \cite[\S6.2]{Serre}.
%\cite[Propositions 10, 11]{Serre}.
For each irreducible character $\psi \in {\rm Irr}(G)$,
fix a representation $\rho^\psi: G\to{\rm GL}_{n_{\!\psi}}\!(\C)$,
where ${\rm GL}_{n_{\!\psi}}\!(\C)$ is the general linear group of degree $n_\psi$ over $\C$,
 such that $\rho^\psi$ affords $\psi$.
For $\alpha \in G$, the representation is expressed as the invertible
 $n_\psi\!\times\! n_\psi$-matrix
$\rho^\psi(\alpha) = \big(\rho^\psi_{ij}(\alpha)\big)_{n_\psi \times n_\psi}$.
Let $\widehat{G}$ denote the set of all matrix entry
functions $\rho^\psi_{ij} \in \C^G$ for $\psi$ running over ${\rm Irr}(G)$, i.e.,
\begin{align} \label{eq wh G=}
\wh G=\big\{\,\rho^\psi_{ij}\;\big|\; 1\le i,j\le n_\psi,\,\psi\in{\rm Irr}(G)\big\}.
\end{align}
This set forms an orthogonal (but not orthonormal) basis for $\C^G$.
Similarly to Equation \eqref{eq wh f=},
the Fourier transform $\widehat{f} \in \C^{\widehat{G}}$ of $f \in \C^G$ is defined by
\begin{align} \label{eq wh f='}
\widehat{f}(\rho^\psi_{ij}) = \sum_{\alpha \in G} f(\alpha) \rho^\psi_{ij}(\alpha),
\quad \mbox{for any}~\rho^\psi_{ij} \in \widehat{G}.
\end{align}
However, the size $\big|{\rm supp}(\widehat{f})\big|$ depends on the choice of representations $\rho^\psi$.
To address this, Meshulam \cite{Me92} introduced the rank support of $\widehat{f}$, defined as
\begin{align} \label{eq rk-supp}
\mbox{rk-supp}(\wh f)\!=\!
\sum_{\psi\in{\rm Irr}(G)}\! n_\psi\!\cdot\!{\rm rank}\big(\wh f(\rho^\psi)\big),
\end{align}
where $\widehat{f}(\rho^\psi) = \big(\widehat{f}(\rho^\psi_{ij})\big)_{n_\psi \times n_\psi}$.
Meshulam proved (\cite{Me92}, see also \cite{WW21}) that, for any nonzero $f \in \C^G$,
\begin{align} \label{u p Me92}
 |{\rm supp}(f)|\cdot\mbox{rk-supp}(\wh f) \ge |G|,
\end{align}
with equality conditions analogous to the abelian case.

Let $G$ be a finite group and $\F$ a field with characteristic
${\rm char}\,\F$ satisfying
\begin{align} \label{condition semisimplicity}
{\rm char}\,\F = 0 \quad \text{or} \quad \gcd\big({\rm char}\,\F, |G|\big) = 1.
\end{align}
Under these conditions, the group algebra $\F G$ is semisimple.
Then the set $\wh G$ is constructed by extending the field $\F$ to a splitting field $\E$,
following the method of Equation \eqref{eq wh G=}.
This approach ensures the framework from Equations \eqref{eq wh f='} to \eqref{u p Me92} remains valid.

Goldstein, Guralnick, and Isaacs \cite{GGI05}
worked within a much broader context as follows. Let $G$ be a finite group, let $X$ be a transitive $G$-set
(cf. \cite[\S3]{AB}, or Remark \ref{rk G-set G-space}(1) below for a definition),
and let $\F$ be a field. Then the vector space $\F X$ with basis $X$ is an $\F G$-module
(called the permutation module, cf. Equation \eqref{eq alpha sum f(x) x} below). Each function $f\in\F^X$ is identified with the element
$\sum_{x\in X}f(x)x\in\F X$. %Inspired by Equation \eqref{eq dim FGf},
Instead of ${\rm supp}(\wh f)$, they consider
the dimension $\dim\F G f$
of the $\F G$-submodule $\F Gf$ of $\F X$ generated by $f$,
and proved that
\begin{align} \label{u p X GGI}
 \big|{\rm supp}(f)\big|\cdot\dim \F Gf\ge|X|, \quad
\mbox{for any nonzero $f\in\F^X$}.
\end{align}
The conditions under which equality holds in
Equation~\eqref{u p X GGI} were also established.
Moreover, they proved that, if $X$ is a transitive $G$-set
with prime cardinality $|X|$ and ${\rm char}\,\F=0$, then
$$
\big|{\rm supp}(f)\big|+\dim\F Gf\ge |X|+1,\quad
\mbox{for any nonzero $f\in\F^X$}.
$$
In particular, if $X=G$ is the left regular set, then $G$ is a
cyclic group of prime order, Equation \eqref{eq dim FGf} directly yields the strong uncertainty principle stated in Equation \eqref{s u p Tao}.

In this paper we consider the uncertainty principles for any field $\F$, any finite group $G$
and any transitive $G$-set $X$.
This paper consists of two parts: firstly we establish
a Fourier transformation over $\F$ on the $G$-set $X$, and show that
for any nonzero $f\in\F^X$ the rank support of the Fourier transform $\wh f$
equals to the dimension $\dim\F Gf$ of
the submodule $\F Gf$ of $\F X$ generated by $f$.
The second part studies the uncertainty principles by the trade-off
between $|{\rm supp}(f)|$ and $\dim\F Gf$.

In Section 2, for any field $\F$ and any finite group $G$,
we introduce necessary preliminaries
about $G$-sets, function spaces and permutation modules.
Then, under the conditions
in Equation \eqref{condition semisimplicity}, we construct
$\wh G$ consisting of the representation matrix entry functions $\rho^\psi_{ij}$
as in Equation \eqref{eq wh G=}.
For any transitive $G$-set $X$, based on $\wh G$,
we construct the dual set $\widehat{X}$
which forms a specific basis of the
function space $\E^X$ ($\E$ is a splitting field for $G$ over $\F$).
Then, for $f\in\F^X\subseteq\E^X$, the Fourier transform
$\wh f\in\E^{\wh X}$ and its rank support
$\mbox{rk-supp}\big(\widehat{f}\,\big) $ are defined.
In particular, when $X = G$ is the left regular $G$-set, these recover
the standard Fourier transform $\wh f$ for $G$ and the rank support
as described in Equations \eqref{eq wh f='} and \eqref{eq rk-supp}.
This framework extends Equation \eqref{eq dim FGf}
to our general setting (see Lemma \ref{lem rk-supp = dim} below):
\begin{align} \label{eq rk-supp=dim}
 \mbox{rk-supp}\big(\widehat{f}\big) = \dim\F Gf, \quad
\mbox{for any}~ f\in\F^X,
 \end{align}
 where $\F Gf$ is the $\F G$-submodule of the permutation
 module $\F X$ generated by~$f$.

Section 3 focuses on any transitive $G$-set $X$, any field $\F$ and $0\ne f\in\F^X$.
We extend the sharpened uncertainty principle by replacing
$\mbox{rk-supp}(\widehat f)$ with $\dim\F Gf$.
This modification removes the field conditions in \eqref{condition semisimplicity}.
It allows the result to hold for a wider range of fields
while preserving the structure of the original principle.
Once this extension is achieved,
the sharpened or unsharpened versions of the uncertainty principle
that utilize $\mbox{rk-supp}(\widehat{f})$
(which retain the conditions in Equation \eqref{condition semisimplicity} on $\F$)
%, or apply to the regular $G$-set $X = G$)
become straightforward corollaries.

Using a natural surjection $G\to X$, we lift the support
$S={\rm supp}(f)$ to a subset ${\scr S}\subseteq G$.
From this, we define the right stabilizer
$G_{\!\scr S}$ of ${\scr S}$ in $G$, which is a subgroup of $G$.
By reducing $G_{\!\scr S}$ to $X$, we obtain the block $X_{S}$ of $X$
(see Remark~\ref{rk block} for the precise definition of blocks),
such that~$S$ is a disjoint union of some translations of the block $X_S$.
Let ${\scr S}'$ denote the complement of ${\scr S}$ in $G$, and
${\scr S}'^{-\!1}=\{\al^{-\!1}\,|\,\alpha\!\in\!{\scr S}'\}$.
By $\F{\scr S}'^{-\!1}\!f$ we denote the subspace of $\F X$ spanned by the
 subset ${\scr S}'^{-\!1}\!f=\{\alpha f\,|\,\al\in{\scr S}'^{-\!1}\}$ of $\F X$.
In particular, $\F Gf$ is the subspace spanned by the subset $Gf$,
which is exactly
the submodule of $F X$ generated by $f$.
We then establish the following sharp uncertainty principle
 in Theorem~\ref{thm act s uncertainty dim} and its Corollary~\ref{cor act s uncertainty dim}
 as follows: %$\dim\F Gf-\dim\F{\scr S}'^{-\!1}f\ge 1$ and
\begin{equation*}
\begin{split}
\big|{\rm supp}(f)\big|\!\cdot\! \dim\F Gf
  &\ge\big|X\big|+\big(\dim\F Gf-\dim\F{\scr S}'^{-\!1}\!f\big)
  \!\cdot\! \big|{\rm supp}(f)\big|  -\big|X_{{\rm supp}(f)}\big|\\[2pt]
  &\geq\big|X\big|+
  \big|{\rm supp}(f)\big|
  -\big|X_{{\rm supp}(f)}\big|.
  \end{split}
\end{equation*}
Here, the first inequality becomes equality if and only if $f$ is
an ${\scr S}'^{-\!1}$-linear function (see Definition \ref{def G-linear function} below).
And, both the inequalities become equalities if and only if
$f$ is an ${\scr S}'^{-\!1}$-linear function and $\F f+\F{\scr S}'^{-\!1}f=\F Gf$.

%We include a small example,  Example \ref{example1},
%to demonstrate that the first inequality can be attained as equality,
%while the second inequality
%(the permutation module version of the
% sharpened uncertainty principle) remains strict.

As a consequence, we reobtain Equation \eqref{u p X GGI},
along with the condition for equality in this equation, as stated in Corollary
\ref{cor act uncertainty dim}.
Furthermore, as indicated previously, we derive
several corollaries involving rank support,
denoted as $\mbox{rk-supp}\big(\wh f\,\big)$,
or the regular set  $X=G$.
For instance, when   $X=G$ is considered as the left regular $G$-set,
the rank support version of Corollary~\ref{cor act s uncertainty dim}
presents the following sharpened uncertainty principle for any finite groups:
\begin{align*} %\label{eq intr G s uncertainty supp}
\big|{\rm supp}(f)\big|\cdot\mbox{rk-supp}\big(\wh f\,\big)
  \ge\big|G\big|+\big|{\rm supp}(f)\big|-\big|G_{{\rm supp}(f)}\big|,
\end{align*}
and the condition for equality holding is exhibited
(see Corollary~\ref{cor act s uncertainty supp} below).

In the concluding  part of Section \ref{s u p for X F},
we utilize $\dim\C G f$ instead of $|\mbox{supp}(\wh f\,)|$
to reestablish the strong uncertainty principle given by
Equation \eqref{s u p Tao}.
We provide a  lower bound on
$\dim\C G f$ for groups $G$   of prime order (see Lemma~\ref{roots}).
This leads to a precise characterization of the
conditions under which equality holds in Equation~\eqref{s u p Tao}.

Section~\ref{conclusion} concludes this paper.

%=============================
\section{Fourier transforms for finite group actions}\label{preliminaries}
%=============================

In this paper,
we consider a finite group $G$   of order $|G|=n$ with
multiplication as its operation.
The identity element of $G$ is denoted by $1_G$
or simply $1$.
Let~$\F$ be any field and let $\F^\times$ be
the multiplicative group of all
units (non-zero elements) of $\F$.
Recall that $\F^G$ denotes the $\F$-vector space of all functions
from~$G$ to~$\F.$

\subsection{Function spaces and permutation modules}

We denote the group algebra by $\F G$,
which  is the $\F$-vector space
with basis $G$,
and the multiplication is defined in such a way that it
is consistent with the multiplication in $G$.
The following is a natural linear isomorphism:
\begin{align} \label{eq F^G to FG}
\F^G \longrightarrow\, \F G, ~~~
g\;\longmapsto\,\sum_{\al\in G} g(\al)\al.
\end{align}
For any  $g,h\in\F^G$,
the convolution $g*h\in\F^G$ is defined as follows:
\begin{align} \label{eq g*h}
(g*h)(\al)=\sum_{\be\in G}g(\be)h(\be^{-1}\al),
\quad\hbox{for any $\alpha$}\in G.
\end{align}
It is readily seen  that Equation \eqref{eq F^G to FG}
is an $\F$-algebra isomorphism.
In this way,  we
identify any function $g\in\F^G$ with the element
$g=\sum_{\al\in G}g(\al)\al\in\F G$.

\begin{remark} \label{rk G-set G-space} \rm
(1) A {\em $G$-set} $X$ is a set equipped with a {\em $G$-action},
which is defined by a map $G\times X\to X$ given by $(\al,x)\mapsto \al x$.
This map must satisfy the following properties:
\begin{itemize}
\item[{\rm(i)}] $(\al\be)x=\al(\be x)$ $\hbox{for any}~\,\al,\be\in G$
and  $\,x\in X$;
\item[{\rm(ii)}] $1_G x=x$ $\hbox{for any}~\, x\in X$.
\end{itemize}
A $G$-set $X$ is said to be {\em transitive}
if for any $x, y\in X$ there is an $\al\in G$
such that $\al x=y$.
Note that if we take $X=G$ and define the product $\alpha x$
for $(\alpha, x) \in G \times G$ as the multiplication in $G$,
this structure is clearly a transitive $G$-set.
This choice of $X=G$ is referred to as the {\em left regular} $G$-set.

(2) An $\F G$-module $V$, also referred to as a $G$-space over $\F$,
is defined as an $\F$-vector space with a {\em $G$-action on the space}.
More explicitly, there is a map $G\times V\to V$ given by
 $(\al ,v)\mapsto \al v$,
which satisfies the following properties:
\begin{itemize}
\item[{\rm(i)}] $(\al\be)v=\al(\be v)$  $\hbox{for any}~\al,\be\in G$
and $\,v\in V$;
\item[{\rm(ii)}] $1_Gv=v$ $\hbox{for any}~\,v\in V$;
\item[{\rm(iii)}] the map $\al:V\to V$, given by $v\mapsto \al v$,
is a linear transformation of $V$
 $\hbox{for any}~\,\al\in G$.
\end{itemize}
\end{remark}

Let $X$ be   a finite $G$-set with cardinality $|X|=m$.
Let $\F X$ be the $\F$-vector space with basis $X$.
The group $G$   acts on the vector space $\F X$ in a natural way,
as follows:
\begin{align} \label{eq alpha sum f(x) x}
 \al\Big(\sum\limits_{x\in X}f(x)x\Big)=\sum\limits_{x\in X}f(x)\al x
  \;=\sum\limits_{y\in X}f(\al^{\!-1}y) y, \quad
\end{align}
$\hbox{for any}~\, \al\in G~\hbox{and}~\sum\limits_{x\in X}f(x)x\in\F X.$
Thus, $\F X$ is an $\F G$-module,
referred to as the {\em permutation module} of the $G$-set $X$.

On the other hand, $G$ naturally acts on the function space
 $\F^X$ as well.  For $\al\in G$ and $f\in\F^X$, the action $\al f\in\F^X$
is defined as follows (compare with Equation \eqref{eq alpha sum f(x) x}):
\begin{align} \label{eq alpha f}
 (\al f)(x)=f(\al^{-1} x), \quad \hbox{for any}~\, x\in X.
\end{align}
Therefore, $\F^X$ is also an $\F G$-module.
Similar to the construction in Equation~\eqref{eq F^G to FG},
we have a natural $\F G$-module isomorphism:
\begin{align} \label{eq F^X to FX}
\F^X \longrightarrow\, \F X, ~~~
f\;\longmapsto\,\sum_{x\in X} f(x)x.
\end{align}
In this context,
we associate any function   $f\in\F^X$ with the element
$f=\sum_{x\in X}f(x)x\in\F X$.
Notably, extending the  $G$-action  given in Equation \eqref{eq alpha f},
for $g=\sum_{\al\in G}g(\al)\al\in\F G$ and $f\in\F^X$,  we have
$$
\big(\sum_{\al\in G}g(\al)\al\big)f=\sum_{\al\in G}g(\al)(\al f),
$$
which can be expressed as
$$
\Big(\big(\sum_{\al\in G}g(\al)\al\big)f\Big)(x)=
\sum_{\al\in G}g(\al) f(\al^{-1}x), \quad \hbox{for any}~\, x\in X.
$$
Viewing $g\in\F^G$ as described
in Equation \eqref{eq F^G to FG},
the right-hand side of the above equation defines
the convolution  $g*f\in\F^X$ as follows:
$$
 (g*f)(x)=\sum_{\al\in G}g(\al) f(\al^{-1}x),\quad
\hbox{for any}\,g\in\F^G, \,f\in\F^X\,~\hbox{and}~ x\in X.
$$
Given $f\in\F^X$, the support of $f$ is defined as the following subset of $X$:
\begin{align} \label{eq supp of f}
 {\rm supp}(f)=\{x\;|\;x\in X,\, f(x)\ne 0\}.
\end{align}
For $\al\in G$, since
$(\al f)(x)=f(\al^{-1}x)$  for any $ x\in X$,
we observe that
\begin{align}\label{eq supp(alpha f)}
{\rm supp}(\al f)=\al\!\cdot{\rm supp}(f), \quad \hbox{for any}\,\al\in G~
\hbox{and}~ f\in \F^X,
\end{align}
where $\al\!\cdot{\rm supp}(f)=\{\al y\,|\,y\in {\rm supp}(f)\}$.
%A subset $S\subseteq X$ is said to be
%{\em $G$-stable} if $\al S=S$ for any $\al\in G$.

\subsection{Fourier transform for group actions}

To describe Fourier transformations,
we assume in this section that the field~$\F$
satisfies the conditions specified in
Equation  \eqref{condition semisimplicity},
ensuring that the group algebra $\F G$ is semisimple.
Furthermore,  we extend $\F$ to a splitting field $\E$ for~$G$.

Let $\omega$ be a primitive $\exp(G)$-th root of unity,
where $\exp(G)$ is the exponent of $G$
(i.e., the least common multiple of the orders of the elements of $G$).
Define $\E=\F(\omega)$ as the extension field of $\F$
obtained by adjoining  $\omega$.
Then, $\E$ serves as a splitting field for $G$.
For more details, consult
\cite[Theorem 24]{Serre} for the case when ${\rm char}\kern1pt\F=0$,
and \cite[Proposition 43]{Serre} for the case when
 ${\rm char}\kern1pt\F\ne 0$.
Consequently,
any $\E$-irreducible character $\psi$ is absolutely irreducible.

Let ${\rm Irr}(G)$ be the set of all absolutely irreducible characters of $G$.
For any $\psi\in{\rm Irr}(G)$,
let $n_\psi$   denote the degree of $\psi$.
There exists a representation (homomorphism)
$\rho^\psi: G\to {\rm GL}_{n_{\!\psi}}\!(\E)$
that affords the character $\psi$,
where ${\rm GL}_{n_{\!\psi}}\!(\E)$ is the group of all
invertible matrices of degree  $n_\psi$ over the field $\E$.
Specifically, we express
$$\rho^\psi(\alpha) = \left( \rho^\psi_{ij}
(\alpha) \right)_{1 \leq i,j \leq n_\psi}, $$
where $ \rho^\psi_{ij} \in \mathbb{E}^G $ for $ 1 \leq i,j \leq n_\psi $. The trace of $\rho^\psi(\alpha)$ is
$$
\mathrm{Tr}\big(\rho^\psi(\alpha)\big) = \psi(\alpha).
$$
Let
\begin{align} \label{eq hat G}
\wh G=\{\,\rho^\psi_{ij}\;|\;\psi\in{\rm Irr}(G),\,i,j=1,\cdots,n_\psi\}.
\end{align}
For $\rho^\psi_{ij},\rho^\varphi_{k\ell}\in\wh G$ and $\al,\be\in G$,
we apply the corollaries of Schur's Lemma
(see \cite[\S2.2, Corollaries 2 and 3]{Serre})
to obtain the following formula:
\begin{align} \label{eq rho * rho}
(\rho^\psi_{ij}\!*\!\rho^\varphi_{k\ell})(\al)=
\sum\limits_{\be\in G}\rho^\psi_{ij}(\be^{-1})\rho^\varphi_{k\ell}(\be\al)
 =\begin{cases}
\frac{n}{n_\psi}\rho^\psi_{i\ell}(\al), & \mbox{$\psi=\varphi$ and $j=k$;}\\
0, &\mbox{otherwise.}
\end{cases}
\end{align}
It follows from \cite[Propositions 4, 5 and their corollaries]{Serre} that
$\wh G$ is a linearly independent (and indeed orthogonal)
subset of $\E^G$, such that
\begin{align}
|\wh G|=\sum_{\psi\in{\rm Irr}(G)} n_\psi^2 =n=|G|.
\end{align}
In fact,  $\wh G$  constitutes a basis for $\E^G$.

\begin{definition}\label{def hat G} \rm
We refer to  $\wh G$ in Equation \eqref{eq hat G}
as the {\em dual basis} of $G$.
\end{definition}

By the classical decomposition of $\E G$-modules
(see \cite[\S2.6 Theorem 8]{Serre}),
the space $\E^X$ decomposes into a direct sum
$$\E^X=\bigoplus_{\psi\in{\rm Irr}(G)}V^\psi,$$
where each $V^\psi$, called the $\psi$-component of $\E^X$,
is further decomposed into a direct sum given by
\begin{align} \label{decomp homog comp}
 V^\psi=W^\psi_1\oplus\cdots\oplus W^\psi_{m_\psi}, \quad \psi\in{\rm Irr}(G),
\end{align}
with each $W^\psi_i$ being an irreducible submodule affording the character $\psi$.
For $\al\in G$, the action of $\al$ on $W^\psi_i$
induces a linear transformation represented by the matrix
$$
\rho^\psi(\al)=\big(\rho^\psi_{i'j'}(\al)\big)_{1\le i',j'\le n_\psi}.
$$
Thus
$W^\psi_i$ has an $\E$-basis denoted by
$\lam^\psi_{i1},\cdots, \lam^\psi_{i n_\psi }$, satisfying the relation
\begin{align*} %\label{transf and matrix}
\al\lam^\psi_{ij}=\sum\limits_{k=1}^{n_\psi}\rho^\psi_{kj}(\al)\lam^\psi_{ik},
\quad \psi\in{\rm Irr}(G),\; i=1,\cdots,n_\psi, \; j=1,\cdots,m_\psi.
\end{align*}
In matrix form, we express this relation as
$$
\begin{pmatrix}
\al\lam^\psi_{11} & \cdots & \al\lam^\psi_{1 n_\psi}\\
\vdots & \ddots & \vdots \\
\al\lam^\psi_{m_{\!\psi} 1} & \cdots & \al\lam^\psi_{m_{\!\psi}\! n_{\!\psi}}\\
\end{pmatrix}
=
\begin{pmatrix}
\lam^\psi_{11} & \cdots & \lam^\psi_{1 n_\psi}\\
\vdots & \ddots & \vdots \\
\lam^\psi_{m_{\!\psi} 1} & \cdots & \lam^\psi_{m_{\!\psi}\! n_{\!\psi}}\\
\end{pmatrix}
\begin{pmatrix}
 \rho^\psi_{11}\!(\!\al\!) & \cdots & \rho^\psi_{1 n_\psi}\!(\!\al\!)\\
 \vdots & \ddots & \vdots\\
 \rho^\psi_{n_{\!\psi} 1}\!(\!\al\!) &
   \cdots & \rho^\psi_{n_{\!\psi}\! n_{\!\psi}}\!(\!\al\!)
 \end{pmatrix}.
$$
Let $\lam^\psi=\big(\lam^\psi_{ij}\big)_{1\le i\le m_{\!\psi},\,1\le j\le n_{\!\psi}}$
and $\al\lam^\psi=\big(\al\lam^\psi_{ij}\big)_{1\le i\le m_{\!\psi}, 1\le j\le n_{\!\psi}}$. We concisely express this as
\begin{align} \label{transf and matrix}
\al\lam^\psi =\lam^\psi\! \cdot\! \rho^\psi(\al),
 \quad \hbox{for any}\, \al\in G~\hbox{and}~\,\psi\in{\rm Irr}(G).
\end{align}
We are now prepared to present a specific basis for
the vector space $\E^X$.
\begin{definition}\label{def hat X} {\rm
Let
$$\wh X=\{\lam^\psi_{ij}\;|\;
 \psi\in{\rm Irr}(G),\,1\le i\le m_\psi,\,1\le j\le n_\psi\}
$$
be given as above.
Then $\wh X$  constitutes a basis of $\E^X$.
We refer to $\wh X$ as a {\em $G$-dual set} of the $G$-set $X$.
}
\end{definition}
This set $\wh X$  provides a natural basis for $\mathbb{E}^X$, leveraging the structure of the irreducible characters of $G$ and their corresponding representations.
We denote the $\E$-vector space with basis $\wh X$ by  $\E\wh X$. Then each function $h\in\E^{\wh X}$ is extended linearly to a
linear function on the vector space $\E\wh X$, and vice versa.
For $h\in\E^{\wh X}$ and $\al\in G$,
we have the following relation from Equation \eqref{transf and matrix}:
\begin{align}  \label{eq alpha h}
\al h(\lam^\psi)=h\big(\al^{-1}\lam^\psi\big)
=h\big(\lam^\psi\!\cdot\! \rho^\psi (\al^{-1})\big)
=h\big(\lam^\psi\big)\!\cdot\!\rho^\psi (\al^{-1}).
\end{align}

The vector space $\E^{\wh X}$ is equipped with an $\E G$-module structure, as the following lemma asserts.
\begin{lemma} \label{lem alpha h}
Let the symbols be the same as defined above. Then $\E^{\wh X}$ is a $G$-space,
with the
$G$-action  on the vector space $\E^{\wh X}$ defined as follows
{\rm (}which is simply the entry-wise version of Equation \eqref{eq alpha h}{\rm )}:
\begin{align*} %\label{eq alpha h}
\al h(\lam^\psi_{ij})
=\sum_{k=1}^{n_\psi}h(\lam^\psi_{ik})\rho^\psi_{kj}(\al^{-1}), \quad \hbox{for any}\,~\al\in G,~h\in\E^{\wh X}~ \hbox{and}\,~\lam^\psi_{ij}\in\wh X.
\end{align*}
\end{lemma}
\begin{proof}
It is straightforward to verify that for $\al\in G$,
 the map $h\mapsto\al h$
is a linear transformation of $\E^{\wh X}$, i.e.,
$\al(h_1+h_2)=\al h_1+\al h_2$
 $\hbox{for any}\,h_1,h_2\in\E^{\wh X}$,
and $\al(ch)=c(\al h)$ $\hbox{for any}\,h\in\E^{\wh X}
~ \hbox{and}\,c\in\E$.
For $\al,\be\in G$, $h\in\E^{\wh X}$ and
$\lam^\psi=\big(\lam^\psi_{ij}\big)_{m_\psi\times n_\psi}$,
we have
\begin{align*}
(\al\be)h(\lam^\psi) =h\big(\lam^\psi\big)\rho^\psi\big((\al\be)^{-1}\big)
=\big(h(\lam^\psi)\rho^\psi((\be)^{-1})\big)\rho^\psi((\al)^{-1})
=\al\big(\be h(\lam^\psi)\big),
\end{align*}
giving $(\al\be)h=\al(\be h)$. Thus, by Remark \ref{rk G-set G-space}(2),
 $\E^{\wh X}$ is indeed a $G$-space.
\end{proof}

Based on the $\E G$-module structure on
 $\mathbb{E}^{\widehat{X}}$,
 we now define the Fourier transformation for the $G$-set $X$.

\begin{definition}\label{def Fourier transf} \rm
The map  $\E^X \to\E^{\wh X}$, defined by $f\mapsto \wh f$, is given by
$$
 \wh f(\lam^\psi_{ij})=\sum\limits_{x\in X} f(x)\lam^\psi_{ij}(x),
\quad \psi\in{\rm Irr}(G),~ i=1,\cdots,m_\psi,~ j=1,\cdots,n_\psi.
$$
This map is referred to as the
{\em Fourier transformation} for the $G$-set $X$, and
$\wh f$ is called the {\em Fourier transform} of $f$.
\end{definition}

Analogous to the classical Fourier transform on finite groups, we have the following result.
\begin{lemma} \label{lem E^X cong E^hat X}
The Fourier transform  $\E^X\to\E^{\wh X}$, $f\mapsto\wh f$,
is a $G$-space {\rm (}$\E G$-module{\rm)} isomorphism.
\end{lemma}
\begin{proof}
We enumerate ${\rm Irr}(G)$, $X$,
and the index $\wh X$ in lexicographical order:
\begin{align} \label{eq list Irr}
\begin{array}{rl}
{\rm Irr}(G)=\{\psi_1, \cdots,\psi_r\},\quad &X=\{x_1,\cdots,x_m\},\\[5pt]
\wh X=\{\lam^1_{11}, \cdots, \lam^{k}_{ij},\cdots,\lam^{r}_{m_r n_r}\},
 &\lam^k_{ij}=\lam^{\psi_k}_{ij},~ m_k=m_{\psi_k},~  n_k=n_{\psi_k}.
\end{array}
\end{align}
We then have
\begin{align} \label{eq matrix version of wh f}
\begin{pmatrix}\wh f(\lam^1_{11})\\ \vdots\\ \wh f(\lam^r_{m_r n_r})\end{pmatrix}
=\begin{pmatrix} \lam^1_{11}(x_1) & \cdots & \lam^1_{11}(x_m) \\
 \vdots & \ddots & \vdots \\
\lam^r_{m_r n_r}(x_1) & \cdots & \lam^r_{m_r n_r}(x_m)
\end{pmatrix}
\begin{pmatrix} f(x_1)\\ \vdots \\ f(x_m) \end{pmatrix}.
\end{align}
Thus, the mapping $f\mapsto\wh f$ is a linear map.
Since the set
$\{\lam^1_{11}, \cdots, \lam^{k}_{ij},\cdots,\lam^{r}_{m_r n_r}\}$
forms a basis of $\E^{X}$ (hence $m_1n_1+\cdots+m_rn_r=m$),
the rows of the $m\times m$ matrix $\big(\lam^k_{ij}(x_t)\big)_{m\times m}$
in Equation \eqref{eq matrix version of wh f} are linearly independent,
making the matrix invertible.
We have thus shown that  $f\mapsto\wh f$ is a linear isomorphism.
For $\al\in G$, $f\in\E^X$ and
$\lam^\psi=\big(\lam^\psi_{ij}\big)_{m_\psi\times n_\psi}$
as in Equation \eqref{transf and matrix}, we have
\begin{align*}
\wh{\al f}(\lam^\psi)
&=\sum_{x\in X}(\al f)(x)\,\lam^\psi(x)
=\sum_{x\in X}f(\al^{-1} x)\lam^\psi(x)
=\sum_{y\in X}f(y)\lam^\psi(\al y) \\
&=\sum_{y\in X}f(y)\!\cdot\!(\al^{-1}\lam^\psi)(y)
=\sum_{y\in X}f(y)\!\cdot\!\lam^\psi(y)\!\cdot\!\rho^\psi(\al^{-1})\\
&=\sum_{y\in X}\big(f(y)\lam^\psi(y)\big)\rho^\psi(\al^{-1})
  = \wh f(\lam^\psi)\rho^\psi(\al^{-1})
  =(\al \wh f\,)(\lam^\psi),
\end{align*}
where the first equality holds by Equation \eqref{transf and matrix}
and
the last equality  is derived from Equation \eqref{eq alpha h}.
We conclude that
 $\wh{\al f}=\al\wh f$ for any  $\al\in G$ and $f\in\E^X$.
The Fourier transform   $f\mapsto\wh f$
 is therefore a $G$-space isomorphism.
\end{proof}

\begin{remark} \label{rk if X= then hat X} \rm
(1) If $\F=\C$ (hence $\E=\C$) is the complex number field,
then there exists a $G$-dual set $\wh X$
with improved properties (in particular, orthogonality) such that the Fourier inversion can be defined in a classical way, cf.~\cite{FX22}.

(2) If $X=G$ is the left regular $G$-set,
then we take $\wh X=\wh G$ as usual;
the dual basis $\wh G$ exhibits orthogonality,
see Equation \eqref{eq rho * rho},
allowing for improved results, cf. \cite{FX18}.

(3) Assume that $G$ is abelian.
If $X=G$ is the regular set,
then $\wh X=\wh G$ is the dual group of $G$
(where the choice of $\wh X$ is unique up to rescaling),
and all related results are classical.
If $X$ is a transitive $G$-set, it  reduces to the
regular set of a quotient group of $G$.
Finally, if $X$ is not transitive, $X$ is partitioned into orbits,
and it is reduced to each orbit which is transitive, cf. \cite{FX17}.
\end{remark}

\subsection{The rank support of the Fourier transform $\wh f$}
For $f\in\F^X$, the support ${\rm supp}(f)=\{x\,|\,x\in X,\, f(x)\ne 0\}$
was  defined in Equation \eqref{eq supp of f}.
Since $\F^X\subseteq \E^X$, we consider the Fourier transform $\wh f\in\E^{\wh X}$.
For technical reasons, as discussed in \cite[\S3.2.2]{WW21},
we define
$$
 \big|{\rm supp}(\wh f\,)\big|
 =\sum_{\psi\in{\rm Irr}(G)} n_\psi\!\cdot\big|{\rm supp}(\wh f(\lam^\psi))\big|,
$$
where $\wh f(\lam^\psi)=\big(\wh f(\lam^\psi_{ij})\big)_{m_\psi\times n_\psi}$,
as detailed in  Equation \eqref{transf and matrix}.
It is important to note that the $G$-dual set $\wh X$ is not unique and
it depends on
\begin{itemize}
\item[(C1).] the choices of the dual basis $\wh G$ of $G$;
\item[(C2).] the choices of the decompositions as specified in Equation \eqref{decomp homog comp}.
\end{itemize}

\noindent
Following \cite[Definition 3.6]{WW21}, we define
\begin{align}
\big|\mbox{min-supp}(\wh f)\big|=\min_{\wh X} \big|{\rm supp}(\wh f)\big|,
\end{align}
where the minimum is taken over the possible choices of the
 $G$-dual set $\wh X$.
On the other hand, following the work of
 \cite{Me92} (see Equation \eqref{eq rk-supp}),
we define the rank support as follows:
\begin{align} \label{def rk-supp}
\mbox{rk-supp}\big(\wh f\,\big)
=\sum_{\psi\in{\rm Irr}(G)}n_\psi\!\cdot\!{\rm rank}\big(\wh f(\lam^\psi)\big).
\end{align}
Note that $\mbox{rk-supp}\big(\wh f\,\big)$ is an integer, not a set.

Each change associated with the first choice (C1)
implies the existence of a matrix
$P\in{\rm GL}_{n_\psi}\!(\E)$ such that the matrix
$\wh f(\lam^\psi)$ is modified to $\wh f(\lam^\psi)\!\cdot\!P$.
A change stemming from the second choice (C2)
implies there exists a matrix $Q\in {\rm GL}_{m_\psi}\!(\E)$ such that
the matrix $\wh f(\lam^\psi)\!\cdot\!P$
is transformed to  $Q\!\cdot\!\wh f(\lam^\psi)\!\cdot\!P$.
There are matrices $P,Q$ such that
$$
Q\!\cdot\!\wh f(\lam^\psi)\!\cdot\!P
=\begin{pmatrix}
 1 &  &  & \\
    & 1 &   &\\
         &&\ddots  &
 \end{pmatrix}_{m_\psi\times n_\psi},
$$
which is a diagonal (possibly not square) matrix with $1$s and $0$s along the diagonal.
The number of $1$s on the diagonal equals ${\rm rank}\big(\wh f(\lam^\psi)\big)$.

We summarize the preceding discussion in the following lemma.
\begin{lemma}
Let $\wh X$ be defined as in Definition \ref{def hat X}.
For any $f\in\F^X$ we have
$$
\mbox{\rm rk-supp}\big(\wh f\,\big)=\big|\mbox{\rm min-supp}\big(\wh f\,\big)\big|.
\eqno\qed
$$
\end{lemma}

This lemma indicates that the rank support of the Fourier transform $\widehat{f}$ coincides with the minimum support over the choices of the $G$-dual set   $\widehat{X}$.
Note that
if $X=G$ is the left regular set,
then we take $\wh X=\wh G$ as usual
(cf. Remark \ref{rk if X= then hat X}(2)).
 In this scenario, there is only one choice (C1) available.
 Consequently, the matrix $\wh f(\rho^\psi)$ can only be altered to
 $\wh f(\rho^\psi)\!\cdot\!P$. Thus, we derive only the inequality:
 $\mbox{\rm rk-supp}\big(\wh f\,\big)\le\big|\mbox{\rm min-supp}\big(\wh f\,\big)\big|$,
 as discussed in \cite[Lemma 3.8]{WW21}.

Our goal is to establish a sharp  uncertainty principle which trades off
between $|{\rm supp}(f)|$ and $\mbox{\rm rk-supp}\big(\wh f\,\big)$.
Before proceeding, we explore the relationship between
$\mbox{\rm rk-supp}\big(\wh f\,\big)$ and the $\F$-dimension
$\dim_{\F}\F Gf$ of the $\F G$-submodule $\F Gf$
of the $\F G$-module $\F X$ generated by $f$.

\begin{lemma} \label{lem rk-supp = dim}
For any $f\in\F^X$, let $\dim_\E\E G\wh f$ denote the
$\E$-dimension of the $\E G$-submodule $\E G\wh f$
of $\E^{\wh X}$ generated by $\wh f$,
and let $\dim_{\F}\F Gf$ denote the $\F$-dimension
of the $\F G$-submodule $\F Gf$ of $\F^{X}$ generated by $f$.
Then we have
$$\mbox{\rm rk-supp}\big(\wh f\,\big)
   =\dim_\E\E G\wh f =\dim_{\F}\F Gf.
$$
\end{lemma}
\begin{proof}
Let ${\rm M}_{m\times n}(\E)$ \big(and ${\rm M}_{n}(\E)$, respectively\big)
be the $\E$-space of all
$m\times n$ ($n\times n$, respectively) matrices over $\E$.
We maintain the notation as in Equation \eqref{eq list Irr}.
The representations
$\rho^k=\rho^{\psi_k}$, $k=1,\cdots,r$, induce
an $\E$-algebra isomorphism (see \cite[Proposition 10]{Serre}):
\begin{align} \label{EG to M}
\begin{array}{cccc}
 \rho: & \E G & \mathop{\longrightarrow}\limits^{\cong} &
 {\rm M}_{n_1}(\E)\; \times~\cdots~\times\; {\rm M}_{n_r}(\E),
\\
 & \sum\limits_{\al\in G} g(\al)\al & \longmapsto &
 \Big( \sum\limits_{\al\in G} g(\al)\rho^1(\al), ~
  \cdots, ~ \sum\limits_{\al\in G} g(\al)\rho^r(\al) \Big).
\end{array}
\end{align}
Since $\wh X$ is a basis of~$\E^{X}$, by
denoting $\lam^k=\lam^{\psi_k}$ as in
Equation \eqref{transf and matrix},
we have the following linear isomorphism:
\begin{align} \label{E^hat X to M}
\begin{array}{cccc}
 \xi: & \E^{\wh X} & \mathop{\longrightarrow}\limits^{\cong} &
 {\rm M}_{m_1\times n_1}(\E) \times\cdots\times {\rm M}_{m_r\times n_r}(\E),
\\
 & h & \longmapsto & \big(\; h(\lam^1)\;, ~~~ \cdots, ~~~ h(\lam^r)\; \big).
\end{array}
\end{align}
For $\al\in G$ and $A=(A^1,\cdots,A^r)$
 with $A^k\in {\rm M}_{m_k\times n_k}(\E)$,
 define $\al\!\circ\! A$ as follows:
\begin{align} \label{eq alpha circ A}
 \al\!\circ\! A=\big(A^1\rho^1(\al^{-1}),\;\cdots,\; A^r\rho^r(\al^{-1})\big).
\end{align}
It is straightforward to verify that the right-hand
side of Equation \eqref{E^hat X to M}
forms a $G$-space.
Specifically, we confirm condition (i)
for $G$-spaces in Remark \ref{rk G-set G-space}(2) as follows:
for $\al,\be\in G$,
\begin{align*}
(\al\be)\!\circ\!A
&=\big(\cdots,A^k\rho^k\big((\al\be)^{-1}\big),\cdots\big)
=\big(\cdots, A^k\rho^k\big(\be^{-1}\al^{-1}\big),\cdots\big) \\
&=\big(\cdots, A^k\rho^k(\be^{-1})\rho^k(\al^{-1}),\cdots\big)
=\big(\cdots, \big(A^k\rho^k(\be^{-1})\big)\rho^k(\al^{-1}) ,\cdots\big)\\
&=\al\!\circ\big(\cdots,A^k\rho^k(\be^{-1}),\cdots\big)
=\al\!\circ\!(\be\!\circ\! A).
\end{align*}
Using a similar entry-by-entry computation for
$\al\in G$ and $h\in \E^{\wh X}$,
by Lemma~\ref{lem alpha h} (i.e., Equation \eqref{eq alpha h}),
we have
\begin{align*}
 \xi\big(\al h\big)=\big(\cdots, \al h(\lam^k), \cdots\big)
=\big(\cdots, h(\lam^k)\rho^k(\al^{-1}), \cdots\big)
%=\al\!\circ\!\big(\cdots,  h(\lam^k),\cdots\big)
=\al\!\circ\!\xi(h).
\end{align*}
Thus, $\xi$ in Equation \eqref{E^hat X to M} is an $\E G$-module isomorphism.
By utilizing the isomorphisms in Equations
\eqref{EG to M} and \eqref{E^hat X to M},
alongside the definition in Equation~\eqref{eq alpha circ A},
we write
$$
%\xi\big(\langle\wh f\, \rangle_{\E^{\wh X}}\big)=
\xi(\E G\wh f)=\E G\circ \xi(\wh f)=
\wh f(\lam^1){\rm M}_{n_1}(\E)\times\cdots\times
\wh f(\lam^r){\rm M}_{n_r}(\E).
$$
Thus, we find that
$$
\dim_\E \E G\wh f
%= \dim_\E\big(\E G\circ \xi(\wh f)\big)
=\sum_{k=1}^r \dim_\E\big(\wh f(\lam^k){\rm M}_{n_k}(\E)\big)
=\sum_{k=1}^r n_k\!\cdot\!{\rm rank}(\wh f(\lam^k)).
$$
By the definition from Equation \eqref{def rk-supp}, we have
$\dim_\E \E G\wh f =\mbox{rk-supp}(\wh f\,)$.
It follows from Lemma~\ref{lem E^X cong E^hat X} that
$\dim_\E \E G\wh f =\dim_\E \E Gf $.
Finally, since the $\E$-space $\E G f$ is obtained
from the $\F$-space $\F G f$
by extending the coefficient field, we conclude that
$\dim_\E\E Gf =\dim_\F \F G f$.
\end{proof}

\begin{remark} \label{rk G abelian} \rm
(1) From Equation \eqref{eq F^X to FX},
we have the $\F G$-module isomorphism $\F^X\cong \F X$.
Lemma \ref{lem rk-supp = dim} can be rewritten as (i.e. Equation \eqref{eq rk-supp=dim})
$$
\mbox{rk-supp}\big(\widehat{f}\big) = \dim_\F\F G f, \quad
\mbox{for any}~ f\in\F^X.
$$
Thus we consider the uncertainty principle
on the transitive $G$-set $X$ for any field~$\F$
using $\dim_\F\F G f$ instead of $\mbox{rk-supp}(\wh f\,)$
(hence the conditions in Equation \eqref{condition semisimplicity} on $\F$
are no longer necessary).
When such an uncertainty principle is established,
an uncertainty principle on $\mbox{rk-supp}(\wh f)$
(with the conditions in Equation \eqref{condition semisimplicity} on $\F$),
or on the regular set $X=G$, follows directly.

(2)
If $G$ is abelian, we consider only the regular $G$-set $X\!=\!G$
(see Remark~\ref{rk if X= then hat X}(3)).
 In this case, the analysis simplifies as follows:
\begin{itemize}
\item
$n_\psi\!=\!1$ and $\rho^\psi\!=\!\psi$
for each $\psi\in{\rm Irr}(G)$,
making $\wh G\!=\!{\rm Irr}(G)$  the dual group;

\item
We have $\mbox{rk-supp}(\wh f\,)=|{\rm supp}(\wh f\,)|$
since $\wh f(\lam^\psi)\!=\!\wh f(\psi)$ is a $1\times 1$ matrix
in~Equation \eqref{def rk-supp};

\item
 Equations \eqref{E^hat X to M} and \eqref{eq alpha circ A}
reduce to
\begin{align*}
\begin{array}{cccc}
 \xi:~\E^{\wh G}\ \mathop{\longrightarrow}\limits^{\cong}\;
  \E \times\cdots\times \E,
~~~
  h \,\longmapsto\, \big( h(\psi_1),\; \cdots,\; h(\psi_n) \big);
\end{array}
\end{align*}
and, for $\al\in G$ and $(a_1,\cdots,a_n)\in\E \times\cdots\times \E$,
\begin{align*}
 \al\!\circ\!(a_1,\cdots,a_n)=
\big(a_1\psi_1(\al^{-1}),\cdots, a_n\psi_n(\al^{-1})\big);
\end{align*}
\item
Hence,~
$\xi\big(\E G\wh f \,\big)=\E G\circ \xi(\wh f\,)=
\wh f(\psi_1)\E \times\cdots\times
\wh f(\psi_n)\E $;
\item
By the above equality, Lemma \ref{lem rk-supp = dim} becomes evident.
\end{itemize}

\end{remark}

\section{Sharp uncertainty principle} \label{s u p for X F}

In this section we always assume that
$G$ is a finite group of order $|G|=n$
with multiplication as the operation, $\F$ is any field,
and $X$ is a transitive $G$-set (cf. Remark \ref{rk G-set G-space}(1))
with cardinality $|X|=m>0$.

\subsection{Preparations}

For any subsets ${\scr A}\subseteq G$ and $Y\subseteq X$,
we define several useful notations as follows:
\begin{itemize}
\item The inverse set: ${\scr A}^{-1}=\{\al^{-1}\,|\,\al\in{\scr A}\}$;

\item The action of $\mathscr{A}$  on $Y$: ${\scr A}Y=\{\al y\,|\,\al\in{\scr A},y\in Y\}$;

\item If $ Y = \{ y \}$, we simply write $\mathscr{A} Y = \mathscr{A} y$. If $\mathscr{A} = \{ \alpha \}$, we write $\mathscr{A} Y = \alpha Y$.
\end{itemize}

Additionally, for any subsets ${\scr A},{\scr B}\subseteq G$ and $Y\subseteq X$,
the following hold:
$$
({\scr A}{\scr B})Y={\scr A}({\scr B}Y), \quad
({\scr A}\cup {\scr B})Y=({\scr A}Y)\cup ({\scr B}Y).
$$
It is important to note that the distribution for intersection
``$({\scr A}\cap{\scr B})Y={\scr A}Y\cap{\scr B}Y$''
does not generally hold,
see the note following  Corollary \ref{cor x_0 closed} below.
We denote the difference set by
${\scr A}-{\scr B}=\{\al\,|\,\al\in{\scr A} \mbox{ but } \al\notin{\scr B}\}$.
In particular, $X-Y$  represents  the complement of $Y$ in $X$.

Given any point $x_0\in X$, we note that since the $G$-set $X$ is transitive,
 there exists a natural surjective map:
\begin{align}\label{def zeta_x_0 on G}
\zeta_{x_0}\!:\; G\to X, \quad
 \zeta_{x_0}(\al)=\al x_0,\quad\hbox{for any}\, \al\in G.
\end{align}
For any $Y\subseteq X$, we define the inverse image of $Y$ under $\zeta_{x_0}$ by
$$
\zeta_{x_0}^{-1}(Y)=\{\al\,|\,\al\in G,\,\zeta_{x_0}(\al)\in Y\}.
$$
Let us recall the stabilizer of $x_0$ in $G$:
\begin{align} \label{def stabilizer of x_0}
G_{\!x_0}=\{\al\,|\, \al\in G,\,\al x_0=x_0\},
\end{align}
which forms a subgroup of $G$. This holds because for any
$\al,\be\in G_{\!x_0}$, we have
$(\al\be)x_0=\al(\be x_0)=\al x_0=x_0$,
implying $\al\be\in G_{\!x_0}$.
The subgroup $G_{x_0}$ is referred to as the
{\em stabilizer} of $x_0$ in $G$.
As usual,  the set of left cosets of $G_{\!x_0}$ in $G$
is denoted by
$G/G_{\!x_0}=\{\ga G_{\!x_0}\,|\,\ga\in G\}$.
For any $x\in X$, it follows that
 $\zeta_{x_0}^{-1}(x)=\ga_x G_{\!x_0}$
is a left coset of $G_{\!x_0}$,
where $\zeta_{x_0}(\ga_x)=\ga_x x_0=x$.
Consequently,  $G$ acts by left multiplication on the set $G/G_{\!x_0}$,
giving rise to the bijection:
\begin{align} \label{eq equivalence G-sets}
 G/G_{\!x_0} \longrightarrow X,~~~ \ga G_{\!x_0}\;\longmapsto\;\ga x_0,
\end{align}
which is a $G$-equivalence of $G$-sets (cf. \cite[\S3  Proposition 4]{AB}).

\begin{remark} \label{rk k m} {\rm
From the equivalence given in Equation
\eqref{eq equivalence G-sets}, we assume that
\begin{align*} %\label{eq k m}
 |G_{\!x_0}|=k, \quad |X|=m; \quad \mbox{hence } ~  n=|G|=km.
\end{align*}
By this equivalence, for any subset ${\scr A}\subseteq G$,
we have
$$
{\scr A}\subseteq\zeta_{x_0}^{-1}(\zeta_{x_0}({\scr A}))
=\zeta_{x_0}^{-1}({\scr A}x_0),
$$
which leads to the result:
$$
|{\scr A}|\le k\!\cdot\!|{\scr A}x_0|.
$$
}
\end{remark}
The following lemma characterizes the subsets ${\scr A}\subseteq G$
that achieve the above equation with equality.
\begin{lemma} \label{lem x_0 closed}
Let notation be as above. For any subset ${\scr A}\subseteq G$
the following three statements are equivalent:
\begin{itemize}
\item[{\rm(1)}] $\zeta_{x_0}^{-1}(\zeta_{x_0}({\scr A}))
=\zeta_{x_0}^{-1}({\scr A}x_0) ={\scr A}$.

\item[{\rm(2)}]
 ${\scr A}$ is a disjoint union of some left cosets of $G_{\!x_0}$,
i.e., ${\scr A}G_{\!x_0}={\scr A}$.

\item[{\rm(3)}]
 $|{\scr A}|=k\!\cdot\!|{\scr A}x_0|$.
\end{itemize}
\end{lemma}

\begin{proof}
(1) $\Rightarrow$ (2).
From the equivalence given in Equation
 \eqref{eq equivalence G-sets},
for any $x\in X$,
we have $\zeta_{x_0}^{-1}(x)=\ga_x G_{\!x_0}$ where
$\zeta_{x_0}(\ga_x) =\ga_x x_0 =x$.
Thus, from (1) we get
$$
{\scr A}=\bigcup_{x\in\zeta_{x_0}({\scr A})} \ga_x G_{\!x_0}.
$$
Since the union is taken over distinct left cosets of $G_{x_0}$,
 it follows that $\mathscr{A}$ is a disjoint union of these cosets,
 implying (2).

(2)  $\Rightarrow$ (3).
If ${\scr A}=\bigcup_{i=1}^\ell \ga_i G_{\!x_0}$ is a disjoint union,
then by the equivalence Equation \eqref{eq equivalence G-sets} again,
we have
$${\scr A}x_0=\{\ga_1x_0,\cdots,\ga_\ell x_0\},$$
where
$\ga_1x_0,\cdots,\ga_\ell x_0$ are distinct in $X$.
Thus, we obtain
$|{\scr A}|=k\!\cdot\!\ell=k\!\cdot\!|{\scr A}x_0|$, giving (3).

(3)  $\Rightarrow$ (1).
Since we have established that
 ${\scr A}\subseteq \zeta_{x_0}^{-1}({\scr A}x_0)$, and
since $\zeta_{x_0}^{-1}(x)=\ga_x G_{x_0}$ for $x\in X$,
we find that
$$
|\zeta_{x_0}^{-1}({\scr A}x_0)|=k\!\cdot\!|{\scr A}x_0|=|{\scr A}|,$$
which demonstrates that (1) holds.
\end{proof}

If one of its three statements of Lemma \ref{lem x_0 closed} holds, then we say that
 %confirms one of its three statements, then
 ${\scr A}$ is %considered as
 an {\em $x_0$-closed subset} of $G$.
 The following corollary exhibits various properties of  $x_0$-closed subsets.

\begin{corollary} \label{cor x_0 closed}
Let ${\scr A}$, ${\scr B}$ be subsets of $G$.
The following three statements hold.
\begin{itemize}
\item[{\rm(1)}] If ${\scr A}$ is $x_0$-closed,
then so is ${\scr B}{\scr A}$.

\item[{\rm(2)}] If both ${\scr A}$ and ${\scr B}$ are $x_0$-closed,
then so are ${\scr A}\cup{\scr B}$ and ${\scr A}\cap {\scr B}$.

\item[{\rm(3)}] If both ${\scr A}$ and ${\scr B}$ are $x_0$-closed, then
$$
\big({\scr A}\cap{\scr B}\big)x_0
=({\scr A}x_0)\cap({\scr B}x_0).
$$
\end{itemize}
\end{corollary}
\begin{proof}
(1). Since ${\scr B}{\scr A}G_{x_0}={\scr B}{\scr A}$,
by Lemma \ref{lem x_0 closed}(2), ${\scr B}{\scr A}$ is $x_0$-closed.

(2). If $\ga\in {\scr A}\cap {\scr B}$,
then $\ga G_{x_0}\subseteq{\scr A}$ and $\ga G_{x_0}\subseteq{\scr B}$,
hence $\ga G_{x_0}\subseteq {\scr A}\cap {\scr B}$.
By Lemma~\ref{lem x_0 closed}(2), ${\scr A}\cap{\scr B}$
is $x_0$-closed.

(3).
Assume that ${\scr A}$ and ${\scr B}$
 are disjoint unions of left cosets of $G_{x_0}$ as follows:
\begin{align*}
{\scr A}=\al_1G_{x_0}\cup\cdots\cup \al_i G_{x_0}\cup
\ga_1 G_{x_0}\cup\cdots\cup \ga_j G_{x_0}, \\[2pt]
{\scr B}=\be_1G_{x_0}\cup\cdots\cup \be_{i'} G_{x_0}\cup
\ga_1 G_{x_0}\cup\cdots\cup \ga_j G_{x_0},
\end{align*}
such that
${\scr A}\cap {\scr B}=\ga_1 G_{x_0}\cup\cdots\cup \ga_j G_{x_0}$.
Then we have
\begin{align*}
({\scr A}\cap {\scr B})x_0
=\{\ga_1 x_0,\cdots,\ga_j x_0\}
=({\scr A}x_0) \cap ({\scr B}x_0),
\end{align*}
which completes the proof.
\end{proof}

Note that the condition
``both $\scr A$ and $\scr B$ are $x_0$-closed''
stated in Corollary~\ref{cor x_0 closed}~(3) is necessary.
For example,  if we partition $G_{\!x_0}$ into two subsets ${\scr A}$ and ${\scr B}$,
then
$({\scr A}\cap {\scr B})x_0=\emptyset\ne\{x_0\}=
({\scr A}x_0) \cap ({\scr B}x_0)$.

\begin{remark}\label{rk block} \rm
(1)~ Assume that $\emptyset\neq B\subseteq X$.
If for any $\al,\be\in G$ either $\al B=\be B$ or $\al B\cap\be B=\emptyset$,
then we say that $B$ is a {\em block} of the transitive $G$-set~$X$.
%$\{\al B\;|\; \al\in G\}$ forms a partition of the set~$X$,
%i.e., for any $\al\in G$, either $B\cap(\al B) = B$ or $B\cap(\al B)=\emptyset$,
The single point $\{x\}$ ($x\in X$) and $X$
itself are referred to as the \emph{trivial blocks}. Obviously,
 $B$ is a block if and only if there are $\ga_1,\cdots,\ga_r\in G$
such that
$$
X=(\ga_1 B)\cup \cdots\cup (\ga_r B)
$$
 is a disjoint union (i.e., a partition of $X$), and for any $\al\in G$
 there is a unique index $i$, $1\le i\le r$, such that $\al B=\ga_i B$.

(2)~ Let $x_0\in B\subseteq X$
(implying $G_{\!x_0}\subseteq \zeta_{x_0}^{-1}(B)$).
It is known that $B$ is a block if and only if
$H=\zeta_{x_0}^{-1}(B)$ is a subgroup of $G$.
 If this condition holds, then
%  $B$ is $H$-stable (thus an $H$-set), and 
the partition
$$
X=(\ga_1 B)\cup \cdots\cup (\ga_r B)
$$
corresponds to the disjoint union
$$
G=(\ga_1 H) \cup\cdots\cup (\ga_r H)
$$
of the left cosets of the subgroup $H=\zeta_{x_0}^{-1}(B)$;
see \cite[\S3 Proposition 9]{AB}. %(\cite[p.32 Proposition 9]{AB})
\end{remark}

For any subset ${\scr S}\subseteq G$, following \cite{FHX19}
(where
$G$ is abelian in \cite{FHX19},
but not necessarily  abelian here), 
we define the right stabilizer as
\begin{align} \label{X def right stabilizer}
 G_{\!\scr S}=\{\al\:|\; \al\in G, {\scr S}\al={\scr S}\}.
\end{align}
It follows that $G_{\!\scr S}$ is a subgroup of $G$ because
for $\al,\be\in G_{\!\scr S}$, we have
${\scr S}(\al\be)=({\scr S}\al)\be={\scr S}\be={\scr S}$.
We call $G_{\!\scr S}$ the {\em right stabilizer} of ${\scr S}$ in $G$.

If ${\scr S}$ is an $x_0$-closed subset of $G$,
then ${\scr S}G_{\!x_0}={\scr S}$, see Lemma \ref{lem x_0 closed}(2).
Thus, we have
\begin{align} \label{eq G_x_0 < G_S}
 G_{\!x_0}\le G_{\!\scr S}\le G.
\end{align}

The next lemma reveals that the
image of the subgroup $G_{\!\scr S}\leq G$
(under the surjective map $\zeta_{x_0}$)
is a block of
the transitive $G$-set $X$.
\begin{lemma} \label{lem def X_S}
Let $\emptyset\neq S \subseteq X$, ${\scr S} = \zeta_{x_0}^{-1}(S) \subseteq G$,
${\scr S}'=G-{\scr S}$,  and let $G_{\!\mathscr{S}}$ be the right stabilizer
of $\mathscr{S}$ in $G$ as defined in Equation \eqref{X def right stabilizer}.
Denote
\begin{align*} %\label{eq def X_S}
 X_{\!S} =\zeta_{x_0}(G_{\!\scr S}) =G_{\!\scr S}x_0.
\end{align*}
Then ${\scr S}$ and $G_{\!\scr S}$ are $x_0$-closed,
 $X_S$ is a block of the transitive $G$-set~$X$,
and the following statements  hold.
\begin{itemize}
\item[{\rm(1)}] $|{\scr S}|=k\!\cdot\!|S|$ and
$|G_{\!\scr S}|=k\!\cdot\!|X_{\!S}|$,
 where $k=|G_{\!x_0}|$ was assumed as in Remark~\ref{rk k m}.

\item[{\rm(2)}] There are $\ga_{1},\cdots,\ga_{\ell}\in G$ such that
$$
S=(\ga_1 X_{\!S})\cup\cdots\cup (\ga_{\ell}X_{\!S})
$$
is a disjoint union.
In particular, $|X_{\!S}|$ is a divisor of $|S|$, and
$|S|=|X_{\!S}|$ if and only if $S=\ga X_{\!S}$ forms a block.

\item[{\rm(3)}] $X_{\!S}=\bigcap_{\al\in {\scr S}}\al^{-1}S.$

\item[{\rm(4)}] If $\al'\in{\scr S}'$, then
 $(\al'^{-1}S)\cap X_S=\emptyset$.
 \end{itemize}
\noindent
{\rm (Definition: We call $X_S$ the
 {\em block associated with the subset $S$} of $X$.)}
\end{lemma}
\begin{proof}
By Lemma \ref{lem x_0 closed}(1), ${\scr S}$ is $x_0$-closed,
hence Equation \eqref{eq G_x_0 < G_S} holds,  ensuring that  $x_0\in X_S$ and
$G_{\!\scr S}G_{\!x_0}=G_{\!\scr S}$.
The latter equality implies that $G_{\!\scr S}$ is also $x_0$-closed
(see Lemma~\ref{lem x_0 closed}(2)). We then have
\begin{align*} %\label{eq X_S is a block}
 G_{\!\scr S}=\zeta_{x_0}^{-1}(\zeta_{x_0}(G_{\!\scr S}))=\zeta_{x_0}^{-1}(X_S).
\end{align*}
By Remark \ref{rk block}(2), $X_S$ is a block.

(1).  Since both ${\scr S}$ and $G_{\!\scr S}$ are $x_0$-closed,
by Lemma \ref{lem x_0 closed}(3), (1) holds.

(2). For $\ga\in{\scr S}$, by the definition in Equation
\eqref{X def right stabilizer}
 of $G_{\!\scr S}$, the left coset $\ga G_{\!\scr S}\subseteq {\scr S}$.
In this way, we can find $\ga_{1},\cdots,\ga_{\ell}\in G$ such that
$$
 {\scr S}=(\ga_1 G_{\!\scr S})\cup\cdots\cup(\ga_{\ell} G_{\!\scr S})
$$
is a disjoint union. Therefore, we have
$$
 S={\scr S}x_0=
\big((\ga_1 G_{\!\scr S})\cup\cdots\cup(\ga_{\ell} G_{\!\scr S})\big)x_0
=(\ga_1 G_{\!\scr S}x_0)\cup\cdots\cup(\ga_{\ell} G_{\!\scr S}x_0),
$$
yielding
$$
S=(\ga_1 X_{\!S})\cup\cdots\cup (\ga_{\ell}X_{\!S}).
$$
Additionally, for $1\le i \neq j\le \ell$, we have
$$
(\ga_i X_{\!S})\cap (\ga_j X_{\!S})=
(\ga_i G_{\!\scr S}x_0)\cap (\ga_j G_{\!\scr S}x_0)
=\big((\ga_i G_{\!\scr S})\cap (\ga_j G_{\!\scr S})\big)x_0
=\emptyset x_0=\emptyset,
$$
where the second equality follows from
Corollary \ref{cor x_0 closed}(3) because $\ga_i G_{\!\scr S}$ is also $x_0$-closed.

(3). By definition, $X_S=G_{\!\scr S}x_0$. Let $\be\in G$.
We have
$\be x_0\in G_{\!\scr S}x_0$ if and only if
$\be\in G_{\!\scr S}$ (because $G_{\!\scr S}$ is $x_0$-closed).
By the definition of  $G_{\!\scr S}$ in Equation~\eqref{X def right stabilizer},
 $\be x_0\in G_{\!\scr S}x_0$ if and only if
$\al\be\in{\scr S}$ for any $\,\al\in{\scr S}$.
That is to say, $\be x_0\in G_{\!\scr S}x_0$ if and only if
$\be\in\al^{-1}{\scr S}$ for any $\,\al\in{\scr S}$.
Hence, as $\al^{-1}{\scr S}$ is $x_0$-closed for any $\,\al\in{\scr S}$,
we conclude that
$\be x_0 \in G_{\!\mathscr{S}} x_0$ if and only if
$$
\be x_0 \in \al^{-1} {\scr S} x_0 = \al^{-1} S, ~\text{ for any } \al \in \mathscr{S}.
$$
Thus, we have
$$X_{\!S}=\bigcap_{\al\in {\scr S}}\al^{-1}S.$$

(4). Suppose that $\al'^{-1}S\cap X_S\neq\emptyset$. Then
there is  $\be\in{\scr S}$ such that
(recall that both $\al'^{-1}{\scr S}$ and $G_{\!\scr S}$ are $x_0$-closed)
$$
 \al'^{-1}\be x_0\in\al'^{-1}S\cap X_S
=\big(\al'^{-1}{\scr S}x_0\big)\cap \big(G_{\!{\scr S}}x_0\big)
=\big(\al'^{-1}{\scr S}\cap G_{\!{\scr S}}\big) x_0.
$$
We have $\al'^{-1}\be\in G_{\scr S}$, which implies
 ${\scr S}\al'^{-1}\be={\scr S}$.
Consequently, there exists    $\be_1\in{\scr S}$
such that $\be_1\al'^{-1}\be=\be$.
This leads to $\al'=\be_1\in{\scr S}$,
 a contradiction to the assumption that   $\al'\in{\scr S}'$.
\end{proof}

The next result plays an important role in the decomposition of $X$
into disjoint subsets.
\begin{lemma}\label{X TS neq G}
Let $\emptyset\ne S\,\subsetneqq\, X$ and $S'=X-S$.
Let ${\scr A}\subseteq G$. The following two statements
are equivalent:
\begin{itemize}
\item[\rm(1)] ${\scr A}S\ne X$.

\item[\rm(2)] There exists an $x\in X$ such that ${\scr A}^{-1}x\subseteq S'$.
\end{itemize}
\end{lemma}

\begin{proof}
(1)\;$\Rightarrow$\;(2).
If $X\supsetneqq {\scr A}S=\bigcup_{\al\in {\scr A}}\al S$,
then there is an $x\in X-\bigcup_{\al\in{\scr A}}\al S$.
By De Morgan's law, we have
$$
x\in\bigcap_{\al\in{\scr A}}(X-\al S)
=\bigcap_{\al\in{\scr A}}\al S',
$$
which implies $\al^{-1}x\in S'$ for any $ \al\in{\scr A}$,
thereby proving (2).

(2)\;$\Rightarrow$\;(1). The converse follows
by reversing the argument outlined above.
\end{proof}

By virtue of Lemma \ref{X TS neq G},
we conclude this subsection with the following result which presents a
disjoint decomposition of $X$ into ${\scr S}'^{-1}S$ and $ X_{S}$.
\begin{corollary}\label{X=S'-1S cup X_S}
Let $\emptyset\ne S\subsetneqq X$ and $S'=X-S$.
Let ${\scr S}=\zeta_{x_0}^{-1}(S)$, ${\scr S}'=\zeta_{x_0}^{-1}(S')$ and
$X_{\!S}=G_{\!\scr S}x_0$.
Then
$$
 X=({\scr S}'^{-1}S)\cup X_{S}, \quad
({\scr S}'^{-1}S)\cap X_{\!S}=\emptyset.
$$
\end{corollary}
\begin{proof}
By the assumption, we have
(with $k=|G_{\!x_0}|$ as in Remark \ref{rk k m}):
$$
 G={\scr S}\cup {\scr S}', ~~ {\scr S}\cap {\scr S}'=\emptyset,
   \quad |{\scr S}|=k\!\cdot\!|S|,~|{\scr S}'|=k\!\cdot\!|S'|.
$$
For any $\al\in {\scr S}^{-1}$, set ${\scr A}={\scr S}'^{-1}\cup\{\al\}$;
then $|{\scr A}|>|{\scr S}'|=k\!\cdot\!|S'|$.
For any $x\in X$, by Remark \ref{rk k m} we have
$$
|{\scr A}^{-1}x|\ge |{\scr A}|/k>k\!\cdot\!|S'|/k=|S'|.
$$
Thus,  Lemma~\ref{X TS neq G}(2) fails to hold, or equivalently,
Lemma~\ref{X TS neq G}(1) fails to hold. We then have
\begin{align} \label{X TS}
 X={\scr A} S=\big({\scr S}'^{-1}\cup\{\al\}\big)S
 =({\scr S}'^{-1}S)\cup (\al S),\quad \mbox{for any}~\alpha\in{\scr S}^{-1}.
\end{align}
This implies that
$X-{\scr S}'^{-1}S\subseteq \al S$ for any $\al\in {\scr S}^{-1}$.
By Lemma \ref{lem def X_S}(3), we have
$$
X-{\scr S}'^{-1}S\subseteq \bigcap_{\al\in {\scr S}} \al^{-1}S = X_S;
$$
that is,  $X=({\scr S}'^{-1}S) \cup X_S$.
By Lemma \ref{lem def X_S}(4),   we have
$({\scr S}'^{-1}S)\cap X_S=\emptyset. $ 
%We are done.
\end{proof}

\subsection{Main results}

Recall that the $\F G$-module $\F^G$ is naturally isomorphic
 to
the left regular $\F G$-module $\F G$,
cf. Equation \eqref{eq F^G to FG}. Under this, any function $g\in\F^G$ is identified with the element
$g=\sum_{\al\in G}g(\al)\al\in\F G$.
Analogously,  the $\F G$-module $\F^X$ is isomorphic in a natural way to
the permutation $\F G$-module $\F X$,
cf. Equation \eqref{eq F^X to FX}.
Here,  any function $f\in\F^X$ is identified with the element
$f=\sum_{x\in X}f(x)x\in\F X$.
The surjective map defined in Equation \eqref{def zeta_x_0 on G}
induces another surjective linear map (which we denote  by $\zeta_{x_0}$ again) given by
\begin{align}\label{def zeta_x_0 on FG}
\zeta_{x_0}\!:~\F G\to\F X, \quad
\zeta_{x_0}\Big(\sum\limits_{\al\in G}g(\al)\al\Big)
=\sum\limits_{\al\in G} g(\al) \al x_0, ~~\hbox{for any}\, g\in\F^G.
\end{align}
The map is a surjective $\F G$-module homomorphism because
for $\be\in G$, $g\in\F^G$, we have
$$
\zeta_{x_0}\Big(\be\sum\limits_{\al\in G}g(\al)\al\Big)
=\zeta_{x_0}\Big(\sum\limits_{\al\in G}g(\al)\be\al\Big)
=\sum\limits_{\al\in G}g(\al)\be\al x_0
=\be\zeta_0\Big( \sum\limits_{\al\in G}g(\al)\al\Big).
$$

Let $0\ne f\in\F^X$ and $S={\rm supp}(f)\subseteq X$. For $\al,\be\in G$,
if $\al S\ne \be S$, it follows from Equation \eqref{eq supp(alpha f)} that
${\rm supp}(\al f)=\al S\neq\be S={\rm supp}(\be f)$. Hence
$\alpha f $ and $\be f$ are linearly independent
(i.e. there is no $0\neq c\in \F$ such that $\be f=c\cdot \al f$).

\begin{definition} \label{def G-linear function} \rm
Assume that $0\ne f\in\F^X$, $S={\rm supp}(f)$
and $\emptyset\neq{\scr A}\subseteq G$.
We say that $f$ is an {\em ${\scr A}$-linear function} if
for any $\al,\be\in {\scr A}$ the following two statements hold:
	
(i)~ either $\al S=\be S$ or $\al S\cap\be S=\emptyset$;

(ii)~
$\al f$ and $\be f$ are linearly dependent if $\al S=\be S$.

\noindent
Similarly to Remark \ref{rk block}(1), (i) is equivalent to that
there are $\al_1,\cdots,\al_t\in{\scr A}$ such that
${\scr A}S=\al_1 S\cup\cdots\cup\al_t S$ is a disjoint union,
and any $\al S$ for $\al\in{\scr A}$ must coincide with one of
$\al_1 S, \cdots,\al_t S$.
%For convenience, we call such $S$ an {\em ${\scr A}$-block}.
\end{definition}

If $|{\scr A}|=1$ or $|S|=1$, then it is clear that $f$ is an ${\scr A}$-linear function.

\begin{remark} \label{rk G-linear}  \rm
Consider the special case that ${\scr A}=G$ and $0\ne f\in\F^X$.
Assume that $f$ is a $G$-linear function. Fix $x_0\in S={\rm supp}(f)$.
Definition~\ref {def G-linear function}(i) implies that
$S$ is a block, hence
${\scr S}=\zeta_{x_0}^{-1}(S)$ is a subgroup of $G$
and $G_{\!x_0}\subseteq {\scr S}$, cf. Remark \ref{rk block}(2).
For $\al\in G$, $\al S=S$ if and only if $\al{\scr S}={\scr S}$,
if and only if $\al\in{\scr S}$ (because ${\scr S}$ is a subgroup of $G$).
So, for $\al\in{\scr S}$, by Definition~\ref {def G-linear function}(ii),
$\al f$ and $f$ are linearly dependent, i.e.,
there is a non-zero scalar $c_\al\in\F^\times$ such that $\al f=c_\al f$.
Note that, if $\al\in G_{\!x_0}\le{\scr S}$, then $\al^{-1}\in G_{\!x_0}$ and
$c_\al f(x_0)=\al f(x_0)=f(\al^{-1}x_0)=f(x_0)$,
so that $c_\al=1$ (since $f(x_0)\ne 0$).
For any $\al,\be\in{\scr S}$, we have
$$
 c_{\al\be}f=(\al\be)f=\al(\be f)=\al(c_\be f)
=c_\be (\al f)=c_\be c_\al f ; %= c_\al c_\be f;
$$
i.e., $c_{\al\be}=c_\be c_\al$, for all $\al,\be\in {\scr S}$.
%We see that
%\begin{itemize} \leftskip15pt
 %\item[(\ref{rk G-linear}.\,i)]
 %the map $\eta: {\scr S}\to \F^\times$, $\eta(\be)=c_\be$ for $\be\in{\scr S}$,
 %is a group homomorphism with kernel $\supseteq G_{\!x_0}$,
 %and $\be f=\eta(\be)f$, for each $\be\in{\scr S}$.
%\end{itemize}
Further, for $\be\in {\scr S}$,
$f(\be x_0)=\be^{-1} f(x_0)=c_{\be^{-1}} f(x_0)$ %=\eta(\be^{-1}) f(x_0)$.
Set $c=f(x_0)$ (so, $0\neq c\in\F$),  and
set $\phi(\be)=c_{\be^{-1}}$ %\eta(\be^{-1})$
for $\be\in{\scr S}$.
Then
\begin{itemize} \leftskip15pt
\item[(\ref{rk G-linear}.\,i)] %the map
    $\phi: {\scr S}\to \F^\times$, $\phi(\be)=c_{\be^{-1}}$ for $\be\in{\scr S}$,
   is a group homomorphism with kernel $\supseteq G_{\!x_0}$;
and for $\al\in G$,
$ f(\alpha x_0)=\left\{\!\! \begin{array}{ll} c\,\phi(\alpha), & \al\in{\scr S};\\
	0, & \al\notin {\scr S}. \end{array}\right.$
\end{itemize}
That is exactly the function that makes the equality in %Equation \eqref{u p Me92}
Equation \eqref{u p X GGI} hold.
Conversely, if ${\scr S}$ is a subgroup of $G$ and %one of the
(\ref{rk G-linear}.\,i) %and (\ref{rk G-linear}.\,ii)
holds, then it follows from the above argument that $f$ is a $G$-linear function.
\end{remark}

\begin{remark}\label{rk F scr A f} \rm
Further, for any subset ${\scr A}\subseteq G$ and $f\in \F^X\cong \F X$,  denote
\begin{align*} %\label{eq scr A f=}
   {\scr A}f=\big\{\,\alpha f\,\big|\,\alpha\in {\scr A}\,\big\}
	~~~\mbox{(which is a subset of $\F X$). }
\end{align*}
Let ``$\F{\scr A}f$'' denote the subspace of $\F X$ spanned by the subset ${\scr A}f$, and let
``$\dim\F{\scr A}f$'' denote the dimension of the subspace.
Denote $\F f=\F\{f\}$ for short.
In particular, $\F Gf$ is the subspace of $\F X$ spanned by the subset $Gf$,
which is exactly the $\F G$-submodule of $FX$ generated by~$f$.
It is clear that $\F{\scr A}f$ is a subspace of $\F Gf$.
\end{remark}

We are now ready to state and prove a key lemma.
If we take ${\scr A}=G$ in the lemma,
then $G\,{\cdot}\,{\rm supp}(f)=X$ and
Equation \eqref{u p X GGI} is reobtained.

\begin{lemma}\label{lem act s uncertainty dim}
Let $\F$ be any field, $G$ be any finite group,
$X$ be any transitive $G$-set and $0\ne f\in\F^X$.
Let $\emptyset \neq{\scr A}\subseteq G$. Then the following inequality holds:
\begin{align} \label{eq local uncertainty}
	\big|{\rm supp}(f)\big|\cdot \dim\F{\scr A}f
	\ge\big|{\scr A}{\cdot}\,{\rm supp}(f)\big|.
\end{align}
The inequality becomes an equality if and only if
$f$ is an ${\scr A}$-linear function.
\end{lemma}

\begin{proof}
Denote $S={\rm supp}(f)$. Take
$\al_1\in{\scr A}$. If $\al_1 S\subsetneqq{\scr A}S={\scr A}\!\cdot\kern1pt{\rm supp}(f)$,
then choose $\al_2\in {\scr A}$ such that
$\al_2 S\;{\not\subseteq}\;\al_1 S$.
Iteratively, we obtain $\al_1,\cdots,\al_{t}\in{\scr A}$
with $t\ge 1$ such that
\begin{align} \label{AS = S^-1 S}
	\begin{array}{l}
		\al_1S\cup\cdots\cup \al_{t}S={\scr A}S, \quad \mbox{and}\\[2mm]
		\al_iS~{\not\subseteq}~ \al_1S\cup\cdots\cup \al_{i-1}S, \quad i=2,\cdots, t.
	\end{array}
\end{align}
Note that ${\rm supp}(\al f)=\al S$, for each $\alpha\in G$,
cf. Equation~\eqref{eq supp(alpha f)}. For any $\alpha\in{\scr A}$,
${\rm supp}(\alpha f)\subseteq{\scr A}S$.
For $1<i\le t$,  the support of $\al_i f$ is not contained in
the union of the supports of $\al_{i-1}f,\cdots,\al_1 f$, see
Equations \eqref{AS = S^-1 S}; % and  \eqref{X= union alpha S};
we see that $\al_i f$ cannot be expressed as
a linear combination of $\al_{i-1}f,\cdots,\al_1 f$.
Thus the following $t$ elements of $\F {\scr A} f$:
\begin{align*}  %\label{X alpha_t f indenpendent}
	\al_1 f,~\cdots,~  \al_t f,
\end{align*}
are linearly independent.
In particular, the following inequality holds:
\begin{align*} %\label{eq t<=dim}
	\dim \F{\scr A}f \; \geq\; t.
\end{align*}
Observe that $|\alpha S|=|S|$ for any $\alpha\in G$. We have
\begin{align} \label{AS >= S^-1 S}
	|S|\cdot\dim \F{\scr A}f
	\,\geq\, t\!\cdot\!|S| = \sum_{i=1}^{t}|\al_iS|
	\geq\big|\al_1S\cup\cdots\cup \al_{t}S\big|.
\end{align}
By Equation \eqref{AS = S^-1 S},
we get the following which is just a rewriting of Equation~\eqref{eq local uncertainty}:
\begin{align}  \label{eq AS<=}
	|S|\cdot\dim \F{\scr A}f\;\geq\; |{\scr A}S|\,;
\end{align}
and the inequality becomes an equality if and only if
the two inequalities in Equation \eqref{AS >= S^-1 S} both become equalities;
equivalently, the following two hold:
	
(\ref{eq AS<=}.\,i)~
${\scr A}S=\alpha_1 S\cup\cdots\cup\alpha_{t}S$ is a disjoint union;

(\ref{eq AS<=}.\,ii)~ $\alpha_1 f, \cdots, \alpha_{t}f$ are a basis of $\F{\scr A}f$.

\smallskip
Assume that the above two hold.
Let $\alpha\in{\scr A}$. By (\ref{eq AS<=}.\,ii),
$\alpha f$ is a linear combination of $\alpha_1 f, \cdots, \alpha_{t}f$;
assume that $\alpha f=c_{i_1}\alpha_{i_1}f+\cdots+ c_{i_k}\alpha_{i_k}f$
with $c_{i_1},\cdots,c_{i_k}\in\F$ ($1\le i_1<\cdots<i_k\le t$) being all non-zero.
By (\ref{eq AS<=}.\,i),
$$
 {\rm supp}(\alpha f)={\rm supp}(c_{i_1}\alpha_{i_1}f+\cdots+ c_{i_k}\alpha_{i_k}f)
 =\alpha_{i_1}S\cup\cdots\cup \alpha_{i_k}S
$$
is a disjoint union. However,  ${\rm supp}(\alpha f)=\alpha S$.
By computing the cardinalities of   both sides, we see that $k=1$,
$\alpha S=\alpha_{i_1}S$.
Thus $S$ satisfies the condition (i) of Definition \ref{def G-linear function}.
Denote $i_1=i$ for short, i.e., $\alpha S=\alpha_i S$, $\alpha f=c_i \alpha _i f$.
If $\beta\in{\scr A}$ is such that ${\rm supp}(\beta f)={\rm supp}(\alpha f)$.
By the same argument as above,
we have a $0\neq d_i\in\F$ such that $\beta f=d_i \alpha_i f$.
Hence $\beta f=d_ic_i^{-1}\alpha f$.
By Definition \ref{def G-linear function},  $f$ is an ${\scr A}$-linear function.
	
Conversely, assume that $f$ is an ${\scr A}$-linear function.
By Definition~\ref{def G-linear function}(i).
there are $\al_1,\cdots,\al_t\in{\scr A}$ such that
${\scr A}S=\al_1 S\cup\cdots\cup\al_t S$ is a disjoint union,
and any $\al S$ for $\al\in{\scr A}$ coincides with one of $\al_1 S, \cdots,\al_t S$.
Because ${\rm supp}(\al_i f)=\al_i S$,
the functions $\al_1 f,\cdots, \al_t f$ are linearly independent. Next,
let $g\in\F{\scr A}f$, i.e., there are $\be_j\in{\scr A}$ and $c_j\in\F$ such that
$g=\sum_{j\in J} c_j\be_j f$, where $J$ is a finite index set.
For each $\be_j$ there is a unique $\al_{k_j}$, $1\le k_j\le t$,
such that $\be_j S=\al_{k_j}S$;
by Definition~\ref{def G-linear function}(ii),
$\be_j f=d_j\al_{k_j}f$ for a $d_j\in\F$.
For $1\le i\le t$, set $J_i=\{ j\in J\,|\, k_j=i\}$.
Then the index set  $J=J_1\cup \cdots\cup J_t$ which is a disjoint union,
and
$$
 g =\sum_{i=1}^t \sum_{j\in J_i} c_j\be_j f
  =\sum_{i=1}^t \sum_{j\in J_i} c_jd_j \al_i f
  =\sum_{i=1}^t \Big(\sum_{j\in J_i} c_jd_j\Big)\al_i f,
$$
i.e., $g$ is a linear combination of $\al_1 f,\cdots, \al_t f$.
In conclusion,  $\al_1 f,\cdots, \al_t f$ are a basis of the subspace $\F{\scr A}f$.
Thus, both (\ref{eq AS<=}.\,i) and (\ref{eq AS<=}.\,ii) hold;
equivalently, Equation \eqref{eq AS<=} becomes an equality.
\end{proof}

Combining Lemma \ref{lem act s uncertainty dim} with Corollary \ref{X=S'-1S cup X_S},
we immediately get the following sharp uncertainty principle, as stated in Introduction.

\begin{theorem} \label{thm act s uncertainty dim}
Let $\F$ be any field, $G$ be any finite group,
$X$ be any transitive $G$-set and $0\ne f\in\F^X$.
Let $S={\rm supp}(f)$ and $S'=X-S$.
Fix an $x_0\in S$, let ${\scr S}=\zeta_{x_0}^{-1}(S)$,
 ${\scr S}'=\zeta_{x_0}^{-1}(S')$, and $X_S$ be as defined in Lemma \ref{lem def X_S}.
Then $\dim\F Gf-\dim\F{\scr S}'^{-1}f\ge 1$
and the following inequality holds:
\begin{align} \label{eq X s uncertainty dim}
  |{\rm supp}(f)|{\cdot} \dim\F Gf\ge
  |X|+\big(\dim\F Gf-\dim\F{\scr S}'^{-1}f\big) {\cdot}|{\rm supp}(f)|
	- |X_{{\rm supp}(f)}|.
\end{align}
The inequality becomes an equality if and only if
$f$ is an ${\scr S}'^{-1}$-linear function.
\end{theorem}

\begin{proof}
Applying Corollary \ref{X=S'-1S cup X_S}, we have
\begin{align} \label{eq |X|=...}
	|X|&=|{\scr S}'^{-1}S|+|X_{\!S}|.
\end{align}
Since we fix $x_0\in S $, i.e., $1_G\in {\scr S}$ (hence $1_G\in{\scr S}^{-1}$),
by Equation \eqref{X TS},
${\rm supp}(f)={\rm supp}(1_G f)=1_G S\;{\not\subseteq}\;{\scr S}'^{-\!1}\!S$,
hence $f\notin\F{\scr S}'^{-\!1}\!f$.
We see that $\F{\scr S}'^{-\!1}\!f$ is a proper subspace of $\F Gf$. Thus
$\dim\F Gf-\dim\F{\scr S}'^{-\!1}\!f\ge 1$.
Further,
\begin{align*}
	&  |X|+\big(\dim\F Gf-\dim\F{\scr S}'^{-1}f\big)\!\cdot\!|S| - |X_{S}| \\
	&= \;
	|S|{\cdot}\dim\F Gf- |S|{\cdot}\dim\F{\scr S}'^{-\!1}\!f +|X|  - |X_{S}|\\
	&=\;
	|S|{\cdot}\dim\F Gf  - |S|{\cdot}\dim\F{\scr S}'^{-\!1}\!f + |{\scr S}'^{-\!1}\!S|
	\qquad (\mbox{by Equation \eqref{eq |X|=...}})\\
	&= \; |S|{\cdot}\dim\F Gf
	- \big(\,|S|{\cdot}\dim\F{\scr S}'^{-\!1}\!f- |{\scr S}'^{-\!1}\!S|\,\big).
\end{align*}
Thus Equation~\eqref{eq X s uncertainty dim} is equivalent to the following:
$$
|S|{\cdot}\dim\F{\scr S}'^{-1}f \,\geq\, |{\scr S}'^{-1}S|.
$$
Applying Lemma \ref{lem act s uncertainty dim}
to the case where  ${\scr A}={\scr S}'^{-1}$,
we complete the proof of the theorem.
\end{proof}

Based on Theorem \ref{thm act s uncertainty dim}, we obtain
the following corollary which is the group-action version of
the sharpened uncertainty principle.

\begin{corollary}\label{cor act s uncertainty dim}
Let $\F$ be any field, $G$ be any finite group, 
$X$ be any transitive $G$-set, and $0\ne f\in\F^X$.
Then the following inequality holds:
\begin{align} \label{eq sharpened}
	\big|{\rm supp}(f)\big|\cdot \dim\F Gf
	\geq\big|X\big|+ \big|{\rm supp}(f)\big|  -\big|X_{{\rm supp}(f)}\big|.
\end{align}
The inequality becomes an equality if and only if
$f$ is an ${\scr S}'^{-1}$-linear function and
$\F f+\F{\scr S}'^{-1}f=\F Gf$,
where ${\scr S}'^{-1}$ is defined in Theorem~\ref{thm act s uncertainty dim}.
\end{corollary}

\begin{proof}
Because $\dim\F Gf-\dim\F{\scr S}'^{-1}f\ge 1$,
from Theorem \ref{thm act s uncertainty dim} we immediately obtain
Equation \eqref{eq sharpened}, where
the inequality becomes equality if and only if
$f$ is an ${\scr S}'^{-1}$-linear function and
$\dim\F Gf-\dim\F{\scr S}'^{-1}f=1$.
Since we fix $x_0\in S $, as argued after Equation \eqref{eq |X|=...},
$f\notin \F{\scr S}'^{-1}f$. Thus,
$\dim\F Gf-\dim\F{\scr S}'^{-1}f=1$ if and only if
$\F f+\F{\scr S}'^{-1}f=\F Gf$.
\end{proof}

We provide a small example to illustrate Theorem \ref{thm act s uncertainty dim}.
The lower bounds established in Theorem \ref{thm act s uncertainty dim}
are sharp because there are examples where these bounds are achieved with equality.

\begin{example}\label{example1} \rm
Consider the symmetric group of degree $3$, denoted by $G=S_3$.
Let $X=\{x_1,x_2,x_3\}$  be the set on which $G$ acts canonically
(take $x_0=x_1$ in Equation \eqref{def zeta_x_0 on G}).
Let $f\in\F^X$ be defined  as
$$
 f(x_1)=1, \qquad f(x_2)=1, \qquad f(x_3)=0.
$$
We have
\begin{align*}
  &~~ S=\,{\rm supp}(f)=\{x_1,x_2\};
  \quad {\scr S}\;=\zeta_{x_1}^{-1}(S)=\{(1), (23), (12),(123)\}; \\
  & {\scr S}'=G-{\scr S}=\{(13),  (132)\}, \quad
   {\scr S}'^{-1}=\{(13),  (123)\};\\
  & G_{\!\scr S}=\{(1),(23)\},  \quad X_S=\{x_1\}; \quad
     {\scr S}'^{-1}S=\big\{x_2,x_3\big\}=(13)S=(123)S.
\end{align*}
Obviously, $S$ and ${\scr S}'^{-1}$ satisfy Definition~\ref{def G-linear function}(i).
Simple algebraic calculations reveal that
$$
 (12)f(x_1)=f\big((12)^{-1}x_1\big)=f(x_2)=1.
$$
Proceeding in this manner,
we compile the following table of function values:
$$
\begin{array}{r|cccccc}
	& ~~ f ~ & (12)f & (13)f & (123)f & (23)f & (132)f  \\ \hline
	x_1  & ~ 1    & 1   &  0    &   0     &  1    & 1 \\
	x_2  & ~ 1   &  1   & 1    &  1      &  0    &  0  \\
	x_3  & ~ 0    &  0    &  1    & 1      & 1   &  1
\end{array}
$$
Because $(13)f=(123)f$, we see that
$$\dim\F{\scr S}'^{-1}f=1, \quad\mbox{and\quad
	$f$ is an ${\scr  S}'^{-1}$-linear function.}
$$
In the following we treat the example in two cases.

(1)~ Assume that the characteristic ${\rm char}\,\F$ is odd or zero.
Then the functions $f, (13)f, (23)f$ are linearly independent functions over $\F$.
This is evident because the corresponding matrix
\begin{align*} %\label{3 matrix}
	\left(\begin{array}{ccc}
		1 & 0 & 1 \\
		1 & 1 & 0 \\
		0 & 1 & 1
\end{array}\right)
\end{align*}
has rank $3$. Consequently, $\dim\F Gf=3$, and
$$
\dim\F Gf-\dim\F{\scr S}'^{-1}f=2.
$$
In conclusion,   Equation \eqref{eq X s uncertainty dim} becomes an equality.
More explicitly, we have
\begin{align*}
& \big|{\rm supp}(f)\big|\cdot \dim\F Gf=2\!\cdot\! 3=6\,,~\mbox{and} \\
%\end{equation*}
%and
%\begin{equation*}
& \big|X\big|+(\dim\F Gf-\dim\F{\scr S}'^{-1}f)
  \big|{\rm supp}(f)\big| -\big|X_{{\rm supp}(f)}\big|=3+2\!\cdot\! 2-1=6.
\end{align*}
However, in this case Equation \eqref{eq sharpened} is
a proper inequality:
$$
|X|+|{\rm supp}(f)| -|X_{{\rm supp}(f)}|
=3+ 2-1=4 < 6 = |{\rm supp}(f)|\cdot \dim\F Gf.
$$
It reveals that  Equation \eqref{eq X s uncertainty dim}
is stronger than Equation \eqref{eq sharpened}.

(2)~
 Assume that ${\rm char}\,\F\!=\!2$. Then the matrix
$\left(\begin{array}{ccc}
	1 & 0 & 1 \\
	1 & 1 & 0 \\
	0 & 1 & 1
\end{array}\right)$ has rank~$2$, hence $\dim\F Gf=2$, and
$\dim\F Gf-\dim\F{\scr S}'^{-1}f=1$.
So Equation \eqref{eq sharpened} becomes an equality:
$$
 |X|+|{\rm supp}(f)| -|X_{{\rm supp}(f)}|
 =3+ 2-1=4=2\cdot 2 = |{\rm supp}(f)|\cdot \dim\F Gf.
$$
\end{example}

The following example (consisting of two small examples)
shows that if one of the (i) and (ii) of Definition \ref{def G-linear function}
is not satisfied then Equation \eqref{eq X s uncertainty dim} is a proper inequality.

\begin{example}\label{example2} \rm
(1)~ Let $G=S_3$, $X=\{x_1,x_2,x_3\}$,
as in Example \ref{example1}. Assume that ${\rm char}\,\F\ne 2,3$.
Take $f\in\F^X$ as follows:
$$
 f(x_1)=1, \qquad f(x_2)=2,\qquad f(x_3)=0.
$$
Fix $x_0=x_1$.  Then $S$, ${\scr S}$, ${\scr S}'$, ${\scr S}'^{-1}$,
$G_{\scr S}$, $X_{S}$ and ${\scr S}'^{-1}S$ are all the same as
described in Example \ref{example1}.
But the function value table is as follows:
$$
 \begin{array}{r|cccccc}
	& ~~ f ~ & (12)f & (13)f & (123)f & (23)f & (132)f  \\ \hline
	x_1  & ~ 1    & 2   &  0    &   0     &  1    & 2 \\
	x_2  & ~ 2   &  1   & 2    &  1      &  0    &  0  \\
	x_3  & ~ 0    &  0    &  1    & 2      & 2   &  1
\end{array}
$$
Though $S$ and ${\scr S}'^{-1}$ satisfy Definition~\ref{def G-linear function}(i),
 $f$ is not an ${\scr S}'^{-1}$-linear function since
$(13)f$ and $(123)f$ are linearly independent
(recall that ${\rm char}\,\F\ne 3$, so
$\left(\begin{array}{ccc} 2 & 1  \\  1 & 2 \end{array}\right)$
has rank~$2$).
It is easy to see that $\dim\F Gf=3$, $\dim\F{\scr S}'^{-1}f=2$.
So
\begin{align*}
	& |S|\cdot\dim\F Gf=2\cdot 3=6;\\
	& |X|+(\dim\F Gf-\dim\F{\scr S}'^{-1}f)\cdot|S|-|X_S|
	=3+(3-2)\cdot 2-1=4.
\end{align*}
In this case Equation \eqref{eq X s uncertainty dim} is a proper inequality.

(2)~ Take $G=\{1,\al,\al^2,\al^3\}$ to be a cyclic group of order $4$,
and $X=G$ to be the left regular set. Let $f\in\F^G$ as
$$
 f(1)=1,\qquad f(\al)=1,\qquad f(\al^2)=0,\qquad f(\al^3)=0.
$$
Fix $x_0=1$. Then
\begin{align*}
	& S={\rm supp}(f)=\{1,\al\},\qquad S'=\{\al^2,\al^3\}, \qquad
	S'^{-1}=\{\al,\al^2\}; \\
	& G_S=\{1\},\qquad X_S=G_S=\{1\}, \qquad
	S'^{-1}S=\{\al,\al^2,\al^3\}.
\end{align*}
Observe that $\al S=\{\al,\al^2\}$ and $\al^2 S=\{\al^2,\al^3\}$.
We see that for $S$ and $S'^{-1}$
Definition~\ref{def G-linear function}(i) does not hold,
hence $f$ is not an $S'^{-1}$-linear function, and
both Equation \eqref{eq X s uncertainty dim} and Equation \eqref{eq sharpened}
are strict inequalities. Indeed, from the following function value table:
$$
\begin{array}{r|cccc}
	& ~~ f ~ & \al f  & \al^2 f & \al^3 f  \\ \hline
	1 ~   &   ~ 1    &    0    &      0     &   1   \\
	\al ~  &   ~ 1    &    1    &      0     &   0    \\
	\al^2  &   ~ 0    &    1    &      1     &  0    \\
	\al^3  &   ~ 0    &    0    &      1     &   1
\end{array}
$$
we can see that $\dim\F Gf=3$, $\dim\F S'^{-1}f =2$. Then we have
\begin{align*}
	& |S|\cdot\dim\F Gf=2\cdot 3=6;\\
	& |X|+(\dim\F Gf-\dim\F{\scr S}'^{-1}f)\cdot|S|-|X_S|
	=4+(3-2)\cdot 2-1=5.
\end{align*}
\end{example}

As mentioned before Lemma \ref{lem act s uncertainty dim},
we recover the result
Equation~\eqref{u p X GGI} in \cite{GGI05} as follows.

\begin{corollary}\label{cor act uncertainty dim}
Let $\F$, $G$, $X$, $f$ and $\F Gf$ be as described
in Theorem \ref{thm act s uncertainty dim}.
Then we have the inequality:
\begin{align} \label{eq uncertainty}
	\big|{\rm supp}(f)\big|\cdot\dim\F Gf \ge \big|X\big|.
\end{align}
Let $S={\rm supp}(f)$, $x_0\in S$, ${\scr S}=\zeta_{x_0}^{-1}(S)$
and $f^{\!\scr S}=f\!\circ\!\zeta_{x_0}\in\F^{\!\scr S}$
(that is, $f^{\!\scr S}\!(\be)=f(\be x_0)$ for all $\be\in{\scr S}$).
The following two statements are equivalent:
\begin{itemize}
\item[{\rm(1)}] The equality in Equation \eqref{eq uncertainty} holds.

%\item[{\rm(2)}]
%${\scr S}$ is a subgroup of $G$
%(i.e., $S$ is a block of the transitive $G$-set $X$),  and
%there exists a group homomorphism $\eta: {\scr S}\to\F^\times$
%with kernel ${\rm Ker}(\eta)\supseteq G_{\!x_0}$ such that
%$\be f=\eta(\be)f$ for all $\be\in{\scr S}$.

\item[{\rm(2)}] ${\scr S}$ is a subgroup of $G$, and
there are an element $c\in\F^\times$ and
 a group homomorphism $\phi: {\scr S}\to\F^\times$
with kernel ${\rm Ker}(\phi)\supseteq G_{\!x_0}$ such that $f^{\!\scr S}=c\phi$, i.e.,
 for $\al\in G$,
$ f(\alpha x_0)=\begin{cases} c\,\phi(\alpha), & \al\in{\scr S};\\
	0, & \al\notin {\scr S}. \end{cases}$
\end{itemize}
\end{corollary}

\begin{proof}
From Lemma \ref{lem act s uncertainty dim}, taking ${\scr A}=G$,
we get that Equation~\eqref{eq uncertainty} holds,
and that it becomes equality if and only if $f$ is a $G$-linear function.
By Remark~\ref{rk G-linear}, $f$ is a $G$-linear function if and only if
${\scr S}$ is a subgroup of $G$ and %one of
(\ref{rk G-linear}.\,i) %and (\ref{rk G-linear}.\,ii)
in Remark~\ref{rk G-linear} holds.
That is, $f$ is a $G$-linear function if and only if %one of (2) and (3)
(2) holds.
\end{proof}

\subsection{Theoretical Consequences and Corollaries}

In this subsection, we consider the left regular $G$-set.
Note that,
in the case that $X=G$ is the left regular $G$-set,
we always take $x_0=1_G$ in Equation \eqref{def zeta_x_0 on G}
(in Equation \eqref{def zeta_x_0 on FG}, resp.), so that  $\zeta_{x_0}$
is the identity map on $G$ (on $\F G$, resp.).
Then, for $S\subseteq X=G$, the symbols ${\scr S}=\zeta_{x_0}^{-1}(S)=S$,
$X_S=\zeta_{x_0}(G_{\scr S})=G_S$ (cf. Lemma~\ref{lem def X_S}
and Equation \eqref{X def right stabilizer}), etc.
The notation in Remark~\ref{rk F scr A f} is adopted.

We begin with
 a group version of the sharp uncertainty principle, which is
a consequence of Theorem~\ref{thm act s uncertainty dim}.

\begin{theorem} \label{thm G s uncertainty dim}
Let $G$ be any finite group, $\F$ be any field and $0\ne f\in\F^G$.
Set $S={\rm supp}(f)$ and $S'=G-S$.
Let $G_{{\rm supp}(f)}$ be defined in Equation \eqref{X def right stabilizer}.
Then $\dim\F Gf-\dim\F S'^{-\!1}\!f\geq 1$ and
\begin{align*} %\label{eq G s uncertainty dim}
|{\rm supp}(f)|{\cdot} \dim\F Gf
  \ge|G|+(\dim\F Gf-\dim\F S'^{-\!1}\!f){\cdot}|{\rm supp}(f)| - |G_{{\rm supp}(f)}|;
\end{align*}
the inequality  becomes an equality if and only if $f$ is an $S'^{-1}$-linear function.
\end{theorem}

\begin{proof} Take a $\beta\in S$ and set $\ga=\be^{-1}$.
Then $1_G\in\ga S={\rm supp}(\ga f)$.
We denote $(\ga S)'=G-\ga S$ as usual.
Applying Theorem~\ref{thm act s uncertainty dim} to $\ga f$,
we get that $\dim\F G(\ga f)-\dim\F (\ga S)'^{-1}(\ga f)\geq 1$ and
\begin{align} \label{group sharp up}
\begin{array}{l} \hskip-18pt
  \big|{\rm supp}(\ga f)\big|{\cdot} \dim\F G(\ga f)\; \ge \\[3pt]
  \hskip-10pt
 \big|G\big|+\big(\dim\F G(\ga f)-\dim\F (\ga S)'^{-1}(\ga f)\big)
 {\cdot}\big|{\rm supp}(\ga f)\big| -  \big|G_{{\rm supp}(\ga f)}\big|,
 \end{array}
\end{align}
which becomes an equality if and only if $\ga f$ is a $(\ga S)'^{-\!1}$-linear function.
It is obvious that \vskip-10pt
\begin{equation} \label{eq supp gamma f}
\begin{array}{c}
|{\rm supp}(\ga f)|=|{\rm supp}(f)|, \quad \F G(\ga f)=\F Gf, \\[3pt]
G_{{\rm supp}(\ga f)}=G_{\ga S} =G_{S}=G_{{\rm supp}(f)}.
\end{array}\end{equation}
Further, \vskip-20pt
\begin{equation} \label{eq gamma S'}
\begin{array}{c}
 (\ga S)'=G-\ga S=\ga(G-S)=\ga S',\quad
  (\ga S)'^{-1}=(\ga S')^{-1}=S'^{-1}\ga^{-1}. \\[3pt]
  \F (\ga S)'^{-1}(\ga f)=\F S'^{-1}\ga^{-1}\ga f =\F S'^{-1}\!f.
\end{array}\end{equation}
%Hence $\F (\ga S)'^{-1}(\ga f)=\F S'^{-1}\ga^{-1}\ga f =\F S'^{-1}f$.
For $\al,\be\in S'^{-1}$, i.e., for $\al\ga^{-1},\be\ga^{-1}\in (\ga S)'^{-1}$,
it is clear that $(\al\ga^{-1})(\ga S)=\al S$,  $(\be\ga^{-1})(\ga S)=\be S$,
$(\al\ga^{-1})(\ga f)=\al f$ and $(\be\ga^{-1})(\ga f)=\be f$.
By Definition~\ref{def G-linear function},
$\ga f$ is a $(\ga S)'^{-1}$-linear function
if and only if $f$ is an $S'^{-1}$-linear function.
Thus this theorem follows from Equation~\eqref{group sharp up}.
\end{proof}

\begin{corollary} \label{cor sharppend uncertainty}
Let notation be as in Theorem~\ref{thm G s uncertainty dim}. Then
$$
|{\rm supp}(f)|{\cdot} \dim\F Gf
  \ge|G|+ |{\rm supp}(f)| - |G_{{\rm supp}(f)}|;
$$
the inequality becomes an equality if and only if $f$ is an $S'^{-\!1}$-linear function
and $\F f+\F S'^{-\!1}\!f=\F Gf$.
\end{corollary}

\begin{proof}
We keep the notation from the proof of Theorem~\ref{thm G s uncertainty dim}.
By Corollary~\ref{cor act s uncertainty dim},
it remains to show that $\F f+\F S'^{-\!1}\!f=\F Gf$
if and only if $\F(\ga f)+\F(\ga S)'^{-\!1}(\ga f)=\F Gf$.
By Equation~\eqref{eq gamma S'},
$\F(\ga S)'^{-1}(\ga f)=\F S'^{-\!1}\!f$, which is a proper subspace of $\F Gf$.
Thus, $\F f+\F S'^{-1}f=\F Gf$ if and only if $\dim\F Gf-\dim \F S'^{-1}f=1$,
if and only if $\F(\ga f)+\F(\ga S)'^{-1}(\ga f)=\F Gf$.
%\begin{align*}
%\F f+\F S'^{-1}f=\F Gf  & ~\iff~ \dim\F Gf-\dim \F S'^{-1}f=1 \\
% & ~\iff~ \F(\ga f)+\F(\ga S)'^{-1}(\ga f)=\F Gf.
%\end{align*}
The proof is completed.
\end{proof}

The next corollary derives from
Corollary~\ref{cor act uncertainty dim}.
% by the fact that parts (2) and (3) of Corollary~\ref{cor act uncertainty dim}
%are equivalent for the regular $G$-set.

\begin{corollary}\label{cor G uncertainty dim}
Let $G$, $\F$, $f$ and $\F Gf$ be defined by Theorem~\ref{thm G s uncertainty dim}.
Then %we have
\begin{align} \label{eq G uncertainty dim}
 \big|{\rm supp}(f)\big|\cdot\dim\F Gf \ge \big|G\big|;
\end{align}
and the following two statements are equivalent:
\begin{itemize}
\item[{\rm(1)}] The equality in  Equation \eqref{eq G uncertainty dim} holds.

\item[{\rm(2)}] The support ${\rm supp}(f)=\ga H$ ($\ga\in G$)
is a left coset of a subgroup $H$ of $G$,
and there exists an element $c\in\F^{\times}\!$ and a
homomorphism $\phi:{H}\to\F^\times$ such that the restriction
$f|_{\ga H}=c\!\cdot\!(\ga\phi)$, i.e.,
$f(\ga\be)=c\!\cdot\!\phi(\be)$ for all $\be\in H$.
\end{itemize}
\end{corollary}
\begin{proof}
Denote $S={\rm supp}(f)$.
Take $\ga\in G$ such that $1_G\in\ga^{-\!1} S={\rm supp}(\ga^{-\!1} f)$.
Applying Corollary~\ref{cor act uncertainty dim} to $\ga^{-\!1} f$, we have
\begin{align} \label{eq G uncertainty dim ga}
 \big|{\rm supp}(\ga^{-\!1} f)\big|\cdot\dim\F G(\ga^{-\!1} f) \ge \big|G\big|.
\end{align}
By Equation~\eqref{eq supp gamma f}, it is exactly the same as
Equation~\eqref{eq G uncertainty dim},
and it becomes an equality if and only if
${\rm supp}(\ga^{-\!1}\! f)=\ga^{-\!1}\! S=H$ is a subgroup of $G$
and there are a $c\in\F^\times$ and a homomorphism
$\phi: H\to\F^{\times}$ such that $(\ga^{-\!1}\! f)|_{H}=c\!\cdot\!\phi$,
i.e., for $\al\in G$,
$\ga^{-\!1}\! f(\al)= \begin{cases} c\!\cdot\!\phi(\al), & \al\in H;\\ 0, & \al\notin H.  \end{cases}$
Note that $\ga^{-\!1}\!f(\ga^{-\!1}\!\al)= f(\ga \ga^{-\!1}\!\al)=f(\al)$.
For $\al\in G$ we have
$
f(\al)=\ga^{-\!1}\!f(\ga^{-\!1}\!\al)
= \begin{cases} c\!\cdot\!\phi(\ga^{-\!1}\!\al), & \ga^{-\!1}\!\al\in H;\\
 0, & \ga^{-\!1}\!\al\notin H.  \end{cases}
$
Note that ${\rm supp}(f)=S=\ga H$ and $\phi(\ga^{-\!1}\!\al)=\ga{\cdot}\phi(\al)$.
We see that
``$(\ga^{-\!1}\! f)|_{H}=c\!\cdot\!\phi$'' is equivalent to
``$f|_{\ga H}=c\!\cdot\!(\ga\phi)$''.
In conclusion, (1) and (2) are equivalent.
\end{proof}

%By taking $X=G$ as the left regular $G$-set
%in Corollary~\ref{cor act uncertainty dim},
%we derive Equation \eqref{eq G uncertainty dim}.
%Notably, the statement (1) corresponds to the equivalence in
%Corollary~\ref{cor act uncertainty dim}~(1).

%Here, when we identify the regular
%$G$-set $G$ with the group $G$ through
%$\zeta_{x_0}$ (as specified in Equation \eqref{def zeta_x_0 on G}),
%we actually set $x_0=1_G$.
%It is possible that $1_G\notin S={\rm supp}(f)$.
%Thus, we can choose $\ga^{-1}\in G$ such that $1_G\in\ga^{-1}S$.
%Therefore,
%$1_G\in{\rm supp}(\ga^{-1}f)$ (cf. Equation \eqref{eq supp(alpha f)}),
%and Corollary~\ref{cor act uncertainty dim}(3) states that
%$H=\ga^{-1}S$ is subgroup
%and $\ga^{-1}f=c\eta$ for some element $c\in\F^\times$
%and  $\eta\in{\rm Hom}(H,\F^\times\!)$. In other words, we have
%$S={\rm supp}(f)=\ga H$ and $f|_S=c\cdot\ga\eta$.
%Consequently, the equivalence of statements
%(1) and (2) follows from the equivalence of
%statements (1) and (3) in Corollary
%\ref{cor act uncertainty dim}.

\begin{remark}\label{rek uncertainty} \rm
If $G$ is a finite abelian group,
and Equation \eqref{condition semisimplicity} holds,
with $\F$ containing a primitive $\exp(G)$-th root of unity,
then $\wh f$ is well-defined, and we have
${\rm supp}(\wh f)=\mbox{rk-supp}(\wh f)=\dim\F Gf$
(see Lemma~\ref{lem rk-supp = dim}).
In this case, for any subgroup $H$ of $G$ and any homomorphism
$\phi: H\to\F^\times$, %(thus $\eta\in\wh H$),
there exists a homomorphism $\chi: G\to\F^\times$ %(i.e.,   $\chi\in\wh G$)
such that the restriction $\chi|_H=\phi$
(see \cite[Ch.6 Proposition 1]{Serre73}).
Consequently, the statement (2) of Corollary \ref{cor G uncertainty dim}
can be rewritten as $f=c'\chi I_{\ga H}$, where $c'=c\,\chi(\ga)^{-1}$ and
$I_{\ga H}(\al)=\begin{cases}1,&\al\in \ga H;\\0, &\al\notin\ga H.\end{cases}$
This is simply the classical result stated after Equation \eqref{u p DS}.

However, when $G$ is non-abelian,  the situation changes.
 In general, for a homomorphism
$\phi: H\to\F^\times$, there may not exist a
homomorphism  $\chi: G\to\F^\times$ such that $\chi|_{H}=\phi$.
For instance,  if $H$ is contained in
the derivative subgroup (or named the commutator subgroup) of~$G$,
then any homomorphism $\chi:G\to \F^\times$ must satisfy
$\chi(x)=1$ for all $x\in H$. This means
$\chi|_H\ne\phi$ if $\phi: H\to \F^\times$
is not a trivial homomorphism.
\end{remark}

%\begin{remark} \rm
As demonstrated in Remark \ref{rk G abelian}(1),
if the condition in Equation \eqref{condition semisimplicity} is satisfied,
then the Fourier transform $\wh f$ and the rank support $\mbox{rk-supp}(\wh f)$
are defined, and any result of this section
can be reformulated with $\mbox{rk-supp}(\wh f)$
instead of $\dim\F G f$. We list the reformulations of
the results in this subsection.
%\end{remark}

The following is a reformulation of
Theorem \ref{thm G s uncertainty dim}.
%and Corollary~\ref{cor G uncertainty dim}, when we take
%$X=G$ to be
%the left regular $G$-set, we derive the following results immediately.

\begin{theorem} \label{thm G s uncertainty supp}
Let $G$ be any finite group and $\F$ be any field
satisfying the conditions in Equation \eqref{condition semisimplicity}.
Let $0\ne f\in\F^G$ and $\wh f$ be the Fourier transform of $f$.
Set $S={\rm supp}(f)$ and $S'=G-S$.
Let $G_{{\rm supp}(f)}$ be defined in Equation~\eqref{X def right stabilizer}.
Then we have that $\dim\F Gf-\dim\F S'^{-\!1}\!f\geq 1$ and
\begin{equation*} \label{eq G s uncertainty supp}
\big|{\rm supp}(f)\big|{\cdot}\, \mbox{\rm rk-supp}({\wh f}\,)
  \ge\big|G\big|+\big(\dim\F Gf-\dim\F S'^{-\!1}\!f\big)
   {\cdot}\big|{\rm supp}(f)\big|-\big|G_{{\rm supp}(f)}\big|,
\end{equation*}
with equality holding if and only if $f$ is an $S'^{-\!1}$-linear function.
\end{theorem}

The following corollary is a reformulation of Corollary~\ref{cor sharppend uncertainty}
which is the sharpened uncertainty principle for any finite groups.

%Theorem~\ref{thm act s uncertainty dim} and
%Corollary~\ref{cor act uncertainty dim}.
%Corollary~\ref{cor act s uncertainty dim}

\begin{corollary}\label{cor act s uncertainty supp}
%Let $G$ be any finite group and $X$ be any transitive $G$-set.
Let $G$, $\F$, $f$, $\wh f$, $S'$ and $G_{{\rm supp}(f)}$
be as in Theorem \ref{thm G s uncertainty supp}.
Then %we have
\begin{equation*} \label{eq act s uncertainty supp}
\big|{\rm supp}(f)\big|\cdot \mbox{\rm rk-supp}({\wh f}\,)
  \ge\big|G\big|+\big|{\rm supp}(f)\big|-\big|G_{{\rm supp}(f)}\big|;
\end{equation*}
and the inequality becomes an equality
if and only if $f$ is an $S'^{-1}$-linear function
and $\F f+\F S'^{-1}f=\F Gf$.
\end{corollary}

%\vspace{0.2cm}
%In the situation where $|{\rm supp}(f)|=|X_{{\rm supp}(f)}|$,
%we have a much simpler characterization.

%\begin{corollary}\label{cor act uncertainty supp}
%Let $G$, $X$, $\F$, $f\in \F^X$ and $\wh f$ be defined as above.
%Then we have
%\begin{equation*} \label{eq uncertainty supp}
%\big|{\rm supp}(f)\big|\cdot \mbox{\rm rk-supp}({\wh f}\,)\ge\big|X\big|.
%\end{equation*}
%Furthermore, equality holds if and only if the statement
%in Corollary~\ref{cor act uncertainty dim}(2) is satisfied.
%\end{corollary}

The next one is a reformulation of Corollary~\ref{cor G uncertainty dim},
which is the classical uncertainty principle for finite groups, see
Equation \eqref{u p Me92}.

\begin{corollary}\label{cor G uncertainty supp}
Let $G$, $\F$, $f\in \F^G$ and $\wh f$
be defined as above. Then
\begin{equation*} \label{eq G uncertainty supp}
 \big|{\rm supp}(f)\big|\cdot\mbox{\rm rk-supp}({\wh f}\,) \ge \big|G\big|;
\end{equation*}
and it becomes an equality if and only if the statement
in Corollary~\ref{cor G uncertainty dim}(2) is satisfied.
\end{corollary}

\subsection{Equality in the strong uncertainly principle}
We conclude this section by presenting an explicit characterization of the functions that achieve equality in the strong uncertainty principle in Equation~\eqref{s u p Tao}.
Let $p$ be a prime number and $\omega$ be a primitive
$p$-th root of unity in the complex field $\C$.
The classical Chebotar\"ev theorem states that all square
submatrices of the Fourier matrix $\big(\omega^{ij}\big)_{0\le i,j\leq p-\!1}$
have non-zero determinants (cf. \cite{Chebotarev,SM96}).

\begin{lemma}\label{roots}
Let $p$ be a prime number and $G=\{1,\al,\cdots,\al^{p-1}\}$ be
a cyclic group of order $p$. Let $0\neq f\in\C^G\cong\C G$,
and $\tilde f(z)=\sum_{i=0}^{p-1}f(\al^i)z^i$ be the complex polynomial in variable $z$
associated with $f$. We have
$$
\big|{\rm supp}(f)\big|+\dim\C Gf\geq p+1,
$$
with equality if and only if the degree $\deg(\gcd(\tilde f(z),z^p-1))=|{\rm supp}(f)|-1$,
where $\gcd(-,-)$ denotes the greatest common divisor.
\end{lemma}

\begin{proof} %Let $\alpha$ be a generator of the cyclic group $G$.
The function $f$ corresponds to the row vector
$(f(\alpha^0),f(\alpha^1),\cdots, f(\alpha^{p-1}))$.
Then $\alpha f$ corresponds to the row vector
$
(f(\alpha^{p-1}),f(\alpha^0),\cdots, f(\alpha^{p-2})).
$
In this way, we see that $\dim\C Gf$ is equal to the rank of the circulant matrix
$$
C=\left(
  \begin{array}{ccccc}
    f(\alpha^0) & f(\alpha^1) & \cdots & f(\alpha^{p-1}) \\
    f(\alpha^{p-1}) & f(\alpha^{0}) & \cdots &  f(\alpha^{p-2}) \\
    \cdots & \cdots & \ddots & \cdots  \\
%    f(\alpha^{2}) & f(\alpha^{3}) & \cdots & f(\alpha^{0}) & f(\alpha^{1}) \\
    f(\alpha^{1}) & f(\alpha^{2}) & \cdots  & f(\alpha^{0}) \\
  \end{array}
\right).
$$
Let $\omega$ be a primitive $p$-th root of unity in the complex number field.
Then the polynomial $z^p-1=\prod_{i=0}^{p-1}(z-\omega^i)$ and
$$
 \deg(\gcd(\tilde f(z), z^p-1))=
 \big|\big\{\,\omega^i\;\big|\; 0\le i\le p-1 \mbox{ and } \tilde f(\omega ^i)=0\,\big\}\big|.
$$
It is well known that the rank of $C$ is equal to
$p-\deg(\gcd(\tilde f(z),z^p-1))$  (for example see \cite{Davis}).

Denote $S={\rm supp}(f)$.
Suppose otherwise that $\deg(\gcd(\tilde f(z),z^p-1))\geq |S|$.
For simplifying notation,
let $|S|=r$ and $S=\{\al^{s_1},\cdots,\al^{s_r}\}$
%be an $r$-subset of $\{0,1,\cdots,p-1\}$
with $0\leq s_1<\cdots<s_r\leq p-1$; i.e.,
$\tilde f(z)=\sum_{k=1}^{r}a_{s_k}z^{s_k}$, where $a_{s_k}=f(\al^{s_k})\ne 0$.
Since it is supposed that $\deg(\gcd(\tilde f(z),z^p-1))\geq |S|=r$,
we can choose
$\omega^{i_1},\omega^{i_2}, \cdots, \omega^{i_r}$,
where $0\leq i_1<\cdots<i_r\leq p-1$,
such that $\tilde f(\omega^{i_j})=0$ for $j=1,\cdots,r$.
%and
%we can assume that $F(z)=\sum_{s\in S}a_sz^s$ with all $a_s$ being non-zero
%complex numbers for $s\in S.$
%Suppose further that
%are roots of $F(z)=0$, where $0\leq i_1<i_2<\cdots<i_r\leq p-1$.
That is,
$$
\tilde f(\omega^{i_j}) %=\sum_{s\in S}a_s\omega^{i_js}
=\sum_{k=1}^ra_{s_k}\omega^{i_js_k}=0, \quad j=1,\cdots,r.
$$
In matrix form,
$$
\left(
  \begin{array}{cccc}
    \omega^{i_1s_1} & \omega^{i_1s_2} & \cdots & \omega^{i_1s_r} \\
   \omega^{i_2s_1} & \omega^{i_2s_2} & \cdots & \omega^{i_2s_r} \\
    \cdots & \cdots & \ddots & \cdots \\
    \omega^{i_rs_1} & \omega^{i_rs_2} & \cdots & \omega^{i_rs_r} \\
  \end{array}
\right)
\left(
  \begin{array}{c}
    a_{s_1} \\
    a_{s_2} \\
    \vdots \\
    a_{s_r} \\
  \end{array}
\right)
=
\left(
  \begin{array}{c}
    0 \\
    0 \\
    \vdots \\
    0 \\
  \end{array}
\right).
$$
By virtue of the aforementioned Chebotar\"ev theorem,
the coefficient matrix is invertible. But $a_{s_1}, \cdots, a_{s_r}$ are non-zero.
That is a contradiction.

It follows that $\deg(\gcd(\tilde f(z),z^p-1))\le |S|-1$. Thus
$$
 \dim\C Gf=p-\deg(\gcd(\tilde f(z),z^p-1))\geq p-|S|+1,
$$
with equality if and only if $\deg(\gcd(\tilde f(z),z^p-1))=|S|-1$.
\end{proof}

Thanks to   Lemma \ref{roots}, we characterize the equality in Equation
\eqref{s u p Tao}.
\begin{theorem}\label{strongequal}
Let $p$ be a prime number and $G$ be a cyclic group of order $p$ generated by $\al$.
Let $0\neq f\in\C^G$ and $\wh f\in\C^{\wh G}$ be the Fourier transform of $f$.
Let $\tilde f(z)=\sum_{i=0}^{p-1}f(\alpha^i)z^i$
be the polynomial in variable $z$ associated with $f$.
Then
\begin{align*}
 \big|{\rm supp}(f)\big| + \big|{\rm supp}(\wh f)\big| \;\geq\; p+1,
\end{align*}
with equality if and only if $\deg(\gcd(\tilde f(z),z^p-1))=|{\rm supp}(f)|-1$.
\qed
\end{theorem}

\section{Conclusion} \label{conclusion}
The uncertainty principle is a fundamental concept
that connects functional behavior across various mathematical areas
and has significant implications in both theoretical and practical
applications of mathematics. In their work,
Feng, Hollmann, and Xiang \cite{FHX19} presented a
sharpened  uncertainty principle applicable to any finite
abelian group  $G$  and for any non-zero function
$0\ne f\in \F^G$
(where $\F$ is a field
such that $\F G$ is semisimple and $\wh f$ is the Fourier transform of $f$
in a splitting field):
\begin{align*} %\label{concl s u p abel}
 \big|{\rm supp}(f)\big|\cdot\big|{\rm supp}(\wh f)\big| \ge
   |G|+|{\rm supp}(f)|-|G_{{\rm supp}(f)}|.
\end{align*}

In this paper, we extend and strengthen this principle
to a broader context:
any transitive $G$-set $X$,
 where $G$ is any finite group,
 and any non-zero function $f\in\F^X$ for any field $\F$.
To this end, we first assume that $\F G$ is semisimple and
construct the $G$-dual set $\wh X$
corresponding to the $G$-set $X$.
We extend the classical Fourier transformation  to accommodate
$G$-actions, resulting in the Fourier transform
 $\wh f\in\E^{\wh X}$ of the function $f\in\F^X$.
% The ``shifting technique''
%presented in \cite{FHX19} is effective for the dual  $\wh G$
%of abelian groups in proving Equation \eqref{concl s u p abel},
%but it does not extend to the dual  $\wh X$ of
% the transitive $G$-set $X$.

Next, we generalize the quantity
 $|{\rm supp}(\wh f)|$
from finite abelian   groups to the concept of rank support,
denoted by $\mbox{rk-supp}(\wh f)$, for group actions.
We additionally demonstrate that
$\mbox{rk-supp}(\wh f)=\dim \F Gf$,
where $\dim\F Gf$ is the $\F$-dimension of
the submodule $\F Gf$ of the permutation module $\F X$ generated by
the element $f=\sum_{x\in X}f(x)x$.
Therefore, we investigate the  uncertainty principle
 with $\dim \F Gf$ instead of $\mbox{rk-supp}(\wh f)$
and establish the following result:
\begin{equation}\label{concl s u p action}
\begin{split}
\big|{\rm supp}(f)\big|{\cdot} \dim\F Gf
  &\ge\big|X\big|+\big(\dim\F Gf-\dim\F{\scr S}'^{-\!1}\!f\big)
  {\cdot}\big|{\rm supp}(f)\big|
  -\big|X_{{\rm supp}(f)}\big|\\[2pt]
  &\geq\big|X\big|+  \big|{\rm supp}(f)\big|
   - \big|X_{{\rm supp}(f)}\big|,
  \end{split}
\end{equation}
 where ${\scr S}'^{-\!1}$ is a subset of $G$ determined by ${\rm supp}(f)$,
 and $\F{\scr S}'^{-\!1}\!f$ denotes the subspace of $\F X$ spanned
 by the subset ${\scr S}'^{-\!1}\!f=\{\alpha f\,|\,\al\in{\scr S}'^{-\!1}\}$ of $\F X$.
We also determine the necessary and
sufficient conditions for these inequalities
to hold as equalities. Furthermore, we explicitly characterize the functions
 $f\in\F^X$ that satisfy the equality conditions.

One advantage of replacing
$\mbox{rk-supp}(\wh f)$ with $\dim \F Gf$
is that $\F$ can be any field,
without the requirement  for
$\F G$ to be semisimple.
Moreover, by utilizing   $\dim \F Gf$,
we apply the ``translating technique''
within the $G$-set $X$ to
establish our sharp uncertainty principle,
as expressed in Equation \eqref{concl s u p action}.
The conditions under which equalities hold in this equation
can also be defined,
allowing us to derive various versions of
the finite-dimensional uncertainty
principle -- both sharpened and classical -- as corollaries.

A third advantage is that, in some instances,
$\dim \F Gf$ is easier to characterize than
$\mbox{rk-supp}(\wh f)$. Indeed, we
provide a  lower bound on
$\dim\C G f$ for groups $G$   of prime order.
This leads to an explicit  characterization of the
conditions under which equality holds in the
strong uncertainty principle.
Explicitly, let~$G$ be a cyclic group of prime order $p$ generated by $\alpha$,
and $0\neq f\in\C^G$. Denote  $\tilde f(z)=\sum_{i=0}^{p-1}f(\alpha^i)z^i$ .
We show that
\begin{align*}
 \big|{\rm supp}(f)\big| + \big|{\rm supp}(\wh f)\big| \ge p+1,
\end{align*}
with equality if and only if $\deg(\gcd(\tilde f(z),z^p-1))=|{\rm supp}(f)|-1$.

%\section*{Acknowledgements}
%The research is supported by NSFC with grant number 12171289.
%The authors would like to thank the editor
%and the anonymous reviewers for their valuable comments
%which helped us to improve the paper greatly.

\end{document}